\documentclass[review,onefignum,onetabnum]{siamart171218}

\usepackage{lipsum}
\usepackage{amsfonts}
\usepackage{graphicx}
\usepackage{epstopdf}
\usepackage{algorithmic}
\usepackage{mathrsfs}
\usepackage{amssymb}
\ifpdf
  \DeclareGraphicsExtensions{.eps,.pdf,.png,.jpg}
\else
  \DeclareGraphicsExtensions{.eps}
\fi

\newsiamremark{remark}{Remark}
\newsiamremark{hypothesis}{Hypothesis}
\crefname{hypothesis}{Hypothesis}{Hypotheses}
\newsiamthm{claim}{Claim}
\newsiamremark{assumption}{Assumption}
\headers{Uniform exponential convergence and asymptotics }{Wenjin Zhang and Yong Li}

\title{Uniform exponential convergence of SAA with AMIS and asymptotics of its optimal value
}

\author{Wenjin Zhang\thanks{College of Mathematics, Jilin University, Changchun, P. R. China.  (\email{zhangwenjinmails@163.com}).}
\and Yong Li \footnotemark[1] \thanks{Center for Mathematics and Interdisciplinary Sciences and School of Mathematics and Statistics, Northeast Normal University, Changchun, P. R. China.
  (\email{liyongmath@163.com}).}  
}

\usepackage{amsopn}

\makeatletter
\newcommand*{\addFileDependency}[1]{
  \typeout{(#1)}
  \@addtofilelist{#1}
  \IfFileExists{#1}{}{\typeout{No file #1.}}
}
\makeatother

\ifpdf
\hypersetup{
  pdftitle={Uniform exponential convergence and asymptotics },
  pdfauthor={Wenjin Zhang and Yong Li}
}
\fi

\begin{document}

\maketitle

\begin{abstract}
 We discuss in this paper uniform exponential convergence of sample average approximation (SAA) with adaptive multiple importance sampling (AMIS) and asymptotics of its optimal value.
Using a concentration inequality for bounded martingale differences, we obtain a new exponential convergence rate.
To study the asymptotics, we first derive an important functional central limit theorem (CLT) for martingale difference sequences.
Subsequently, exploiting this result with the Delta theorem, we prove the asymptotics of optimal values for SAA with AMIS.
\end{abstract}

\begin{keywords}
 uniform exponential convergence, central limit theorem, adaptive multiple importance sampling, martingale difference sequence, concentration inequality.
\end{keywords}

\begin{AMS}
  90C15, 90C26, 90C30
\end{AMS}

\section{Introduction}
A simple optimization problem is shown as follows:
\begin{eqnarray}\label{1}
\displaystyle\min_{x\in{\mathcal{X}}}{E\big[g\big(x,\theta(\omega)\big)\big]},
\end{eqnarray}
where $g:{\mathcal{X}}\times\Theta\rightarrow\mathbb{R}$ is a real valued function with ${\mathcal{X}}\subset{{\mathbb{R}^m}}$ and $\theta:\Omega\rightarrow\Theta\subset\mathbb{R}^r$ is a random variable on $(\Omega,\mathscr{F},P)$.\par
When confronted with complex probability distributions or expectation values that are challenging to compute directly, it becomes particularly important to find an appropriate approximation method.
Indeed, approximation problems are of interest in many fields such as statistics, machine learning and optimization.
One of the most classical methods is sample average approximation (SAA), which is also well-known as the Monte Carlo method.
It has been widely used and intensively studied in various stochastic optimization problems, see $\cite{31,32,33,34}$ for recent relevant literature.
Specifically, the SAA problem corresponding to $\eqref{1}$ is
\begin{eqnarray}
\displaystyle\min_{x\in{\mathcal{X}}}\frac{1}{n}\displaystyle\sum_{i=1}^{n}g(x,\theta_i).
\nonumber
\end{eqnarray}
Here we usually assume that the probability density function in $\eqref{1}$ is independent of $x$,
and $\theta_1,...,\theta_n$ are samples of $\theta(\omega)$.\par
In a slightly more complicated case, where the probability density function depends on $x$, our immediate idea is to convert it to the classical SAA problem.
We rewrite the expectation function in $\eqref{1}$ in the following form:
\begin{eqnarray}\label{2}
{E[g(x,\theta)]}=\displaystyle\int_{\Theta}{g(x,\theta)}{\phi(x,\theta)}{\mathrm{d}{\theta}}\triangleq\displaystyle\int_{\Theta}{F(x,\theta)}{\mathrm{d}{\theta}},
\nonumber
\end{eqnarray}
where ${\phi(x,\theta)}$ is a probability density function and $F(x,\theta)={g(x,\theta)}{\phi(x,\theta)}$.
If there exists a probability density function $\psi$ such that $\psi(\theta)>0$ for all $\theta\in{\Theta}$, then
\begin{eqnarray}\label{3}
{E[g(x,\theta)]}=\displaystyle\int_{\Theta}\frac{{F(x,\theta)}}{\psi(\theta)}{\psi(\theta)}{\mathrm{d}{\theta}}.
\nonumber
\end{eqnarray}
Thus, the classical SAA method is applied to obtain the following form:
\begin{eqnarray}
\frac{1}{n}\displaystyle\sum_{i=1}^{n}\frac{F(x,\theta_{i})}{\psi(\theta_{i})}.
\nonumber
\end{eqnarray}
Such  an approach is called importance sampling (IS). It is worth noting that the quality of the sample approximation depends on $\psi(\cdot)$.
In other words, a poor choice of $\psi(\cdot)$ is sometimes disastrous, while a successful choice of it performs well.
A comprehensive review of the subject about IS and classical SAA can be found in $\cite[Chapter 5]{7}$.
For an overview of SAA for nonconvex stochastic programs, please refer to $\cite[Chapter 10]{35}$.
\par
In this article, we focus on an approximation method that has become very popular in recent years, namely, adaptive multiple importance sampling(AMIS).
It not only retains the advantages of traditional methods, but also dynamically improves the reliability of sampling.
The expression is as follows:
\begin{eqnarray}\label{eqn33}
\frac{1}{n}\displaystyle\sum_{i=1}^{n}\frac{F(x,\theta_{i})}{\psi_{i}(\theta_{i})},
\end{eqnarray}
where each probability density function $\psi_i(\theta_i)>0$ for all $\theta_i\in{\Theta}$.
Its most immediate advantage is that the samples don't have to be either independent nor identically distributed.
The probability density functions are chosen appropriately as the sampling process varies.
This means that the sampling densities $\psi_i$ are multiple and can differ from each other, where $\psi_i$ may depend on $\theta_1,...,\theta_{i-1}$.
Subsequently, $\theta_i$ is sampled from $\psi_i$.

This approach, proposed by Corneut $\cite{15}$, adds adaptive techniques to multiple importance sampling methods.
It is very popular in the field of simulation because it reduces the variance compared with the traditional method and improves the computational efficiency by allocating sampling resources dynamically.
Recently, it has been applied to the derivative-free trust-region optimization algorithm, see $\cite{9}$.
Besides, this method was also used in $\cite{14}$ for integrating geostatistical maps and infectious disease transmission models.

The approximation method we study here is the same as the model in $\cite{2}$, although the research objectives are different.
Therefore, we continue to use the term in $\cite{2}$ for $\eqref{eqn33}$ , namely, SAA with AMIS.
The discussion in this paper is based only on theoretical studies.
Readers interested in algorithms and more information on AMIS are referred to $\cite{18,15,16,17,9,19,14}$.

With the above brief introduction, we now clarify the main research content of this paper.
Consider a stochastic programming problem:
\begin{eqnarray}\label{eqn24}
\displaystyle\min_{x\in\mathcal{X}}\left\{f(x):={E[g\big(x,\theta\big)]}\triangleq\displaystyle\int_{\Theta}{F(x,\theta)}{\mathrm{d}{\theta}}\right\}.
\end{eqnarray}
Its corresponding problem for the SAA with AMIS is as follows:
\begin{eqnarray}\label{eqn25}
\displaystyle\min_{x\in\mathcal{X}}\left\{f_n(x):=\frac{1}{n}\displaystyle\sum_{i=1}^{n}\frac{F(x,\theta_{i})}{\psi_{i}(\theta_{i})}\right\}.
\end{eqnarray}
We investigate two important questions around $\eqref{eqn24}$ and $\eqref{eqn25}$.
One of them is the uniform exponential convergence of $f_n$ to $f$.
The other is the asymptotics with respect to $\eqref{eqn24}$ and $\eqref{eqn25}$, specifically, the asymptotics between the optimal values of $f_n$ and $f$.
For ease of presentation, let $\vartheta$ and $\vartheta_n$ denote the optimal values of $\eqref{eqn24}$ and $\eqref{eqn25}$, respectively.
In addition, let $\mathcal{T}$ and $\mathcal{T}_n$ denote the sets of optimal solutions of $\eqref{eqn24}$ and $\eqref{eqn25}$, respectively.

Actually, when using approximation methods, one is primarily concerned with the relationship between $\vartheta$ and $\vartheta_n$, as well as the relationship between $\mathcal{T}$ and $\mathcal{T}_n$.
Theorems 5.3-5.5 in $\cite{7}$ provide convergence conditions for the optimal value and the optimal solutions of the classical SAA in the independent and identically distributed (iid) case.
Feng et al. proved the convergence results of $\vartheta_n$ and $\mathcal{T}_n$ for SAA with AMIS using the similar approach as in $\cite[Theorem\ 5.3]{7}$, see $\cite{2}$.
The necessary condition therein is the uniform convergence of $f_n$ to $f$, which is established by applying Assumption S-LIP in $\cite{11}$.

Our first goal is to further investigate the uniform exponential convergence of $f_n$ to $f$, which provides the convergence rate.
Uniform exponential convergence has been attracting much attention in numerical methods, optimization and other fields.
There have been a lot of studies on uniform exponential convergence with SAA.
In $\cite{20}$, Dai et al. derived an exponential convergence result and noted the role of large deviation techniques in such problems.
Under the Lipschitz assumption, Homem-de-Mello in $\cite{21}$ achieved a uniform result for non-iid sampling.
Shapiro and Xu addressed the uniform exponential convergence for a class of real valued functions with random variables that satisfy global H\"{o}lder continuity, utilizing Cram\'{e}r's theorem, see $\cite{22}$.
Subsequently, the requirement of global H\"{o}lder continuity has been gradually relaxed, as detailed in Xu's study in $\cite{12}$ and Sun and Xu's discussion in $\cite{24}$.

It is not difficult to see that large deviation techniques have been the usual means of studying uniform exponential convergence of SAA in recent years.
Here, we also consider such techniques in solving similar problems for SAA with AMIS.
The distinction is that we apply a kind of concentration inequality, which has not been found to be used in previous research, instead of Cram\'{e}r's theorem and G\"{a}rtner-Ellis theorem.

On the other hand, in order to achieve the goal above, we also draw on the ideas of econometricians who study stochastic equicontinuity (SE) and generic uniform convergence.
Since generic uniform convergence is one of the important tools for the asymptotics of econometric estimators, they have an in-depth study of such problems.
Newey $\cite{25}$ showed that a necessary condition for generic uniform convergence is the SE condition on a compact set.
Andrew extended the compactness condition to total boundedness under the modified SE condition in $\cite{11}$.
P\"{o}tscher and Prucha $\cite{26}$ clarified the distinctions and relations about equicontinuity-type in different literatures.
They pointed out that the SE condition was introduced in order to use a stochastic version of Ascoli-Arzel\`{a} theorem, essentially a stochastic uniform equicontinuity-type condition, which should be called asymptotically $L_0$ uniformly equicontinuous.
Moreover, it is argued that the arbitrary uniform convergence result on a totally bounded set can be reduced to a uniform convergence result on a compact set in $\cite[Section\ 5]{26}$.
Therefore, we only discuss under the compact set $\mathcal{X}$.
In order to trace the origin, we still use the term SE in this paper.
Indeed, Assumption 3 used in $\cite{2}$ to obtain strong uniform convergence of SAA with AMIS is Assumption S-LIP in $\cite{11}$, which is a sufficient condition for strong SE.
Here, we establish the notion of  adaptive multiple modulus calmness, in order to guarantee the SE of random sequences.

The second goal is to study the asymptotics of the optimal values for SAA with AMIS.
We are not aware of any previous work on this issue.
Here, we are inspired by Shapiro's study on the asymptotics of classical SAA.
The core thought is to use the infinite-dimensional Delta theorem and the suitable functional CLT, see $\cite{27}$ and $\cite[Chapter\ 5]{7}$.
For SAA with AMIS, we show that the associated random sequence is a martingale difference sequence in Lemma $\ref{lem2}$.
Thus, the immediate difficulty we encounter is how to obtain a usable functional CLT for martingale difference sequences.
The breakthrough to solve this problem is that we notice a CLT for martingale difference arrays from $\cite{3}$.
Applying it, we derive the functional CLT for martingale difference sequences, namely Theorem $\ref{cor1}$.
And then,  this goal is achieved in Theorem $\ref{the7}$.

In general, the innovations of this paper are on the following aspects:\par
\textbf{(i)} We prove the uniform exponential convergence for SAA with AMIS by applying the technique of a concentration inequality for bounded martingale differences and the idea of SE.
While uniform exponential convergence has been studied quite deeply and extensively for classical SAA, as we introduced earlier, we are not aware of any previous such results for SAA with AMIS.
This result is shown in Theorem $\ref{the1}$. \par
\textbf{(ii)} We derive a functional CLT for martingale difference sequences (Theorem $\ref{cor1}$), which serves as one of the basic tools for  studying asymptotics in this paper.
This crucial step benefits from a pre-existing result for martingale difference arrays that has not yet been found to be applied in related fields.
This part is presented in \cref{sec4.1}.\par
\textbf{(iii)} We obtain the asymptotics of the optimal values for SAA with AMIS, by using the Delta theorem and the functional CLT for martingale difference sequences that we achieved.
This complements existing theoretical studies on SAA with AMIS.
The details are provided in \cref{sec4.2}, where Theorem $\ref{the7}$ is the primary result.

The rest of the paper is organized as follows.
In \cref{sec2}, we introduce preliminaries such as a concentration inequality for bounded martingale differences and a CLT for martingale difference arrays.
In \cref{sec3}, we investigate uniform exponential convergence of SAA with AMIS.
The results on the functional CLT for martingale difference sequences and asymptotics of the optimal value for SAA with AMIS are shown in \cref{sec4}.
The last section is the conclusion.

\section{Notation and Preliminaries}
\label{sec2}
\subsection{Basic Notation}
Throughout this paper, we adopt the following notation.
An abstract probability space is defined by $(\Omega,\mathscr{F},P)$. $E[\cdot]$ denotes the expectation with respect to the probability measure $P$.
We use $\|\cdot\|$ to denote the Euclidean norm of a vector.
$\stackrel{\mathcal{D}}{\longrightarrow}$ and $\stackrel{P}{\longrightarrow}$ denote convergence in distribution and convergence in probability, respectively.
A sequence $\{X_n\}$ of random variables converges in $L_p$ ($p>0$) to a random variable $X$, if
$
\displaystyle\lim_{n\rightarrow\infty}E\|X_n-X\|^p=0,
$
denoted $X_n\stackrel{\mathcal{L}_p}{\longrightarrow}X$.

$o_p(\cdot)$ stands for a probabilistic analogue of the usual order notation $o(\cdot)$.
That means that $a_n=o_p(b_n)$ if for any $\varepsilon>0$,
$
\displaystyle\lim_{n\rightarrow\infty}{\rm{Prob}}\left(\bigg|\frac{a_n}{b_n}\bigg|>\varepsilon\right)=0,
$
where $\{a_n\}$ and $\{b_n\}$ are sequences of random variables.

The directional derivative of $f: B_1\rightarrow{B_2}$ at $x\in B_1$ in the direction $\mathbf{d}\in B_1$ is defined by
$
f'(x,\mathbf{d}):=\displaystyle\lim_{t\rightarrow0}\frac{f(x+t\mathbf{d})-f(x)}{t},
$
where $B_1$ and $B_2$ are two Banach spaces.

A function $\varpi(x)$ is called a modulus of continuity if it is a strictly monotonic increasing continuous function on $\mathbb{R}_+$ with $\displaystyle\lim_{x\rightarrow{0^+}}\varpi(x)=0$ and
$\displaystyle\limsup_{x\rightarrow{0^+}}\frac{x}{\varpi(x)}<+\infty$.

$\mathbb{D}(A,B)$ denotes the deviation of the set $A$ from the set $B$, where $A, B\subset\mathbb{R}^n$. That is, $\mathbb{D}(A,B):=\displaystyle\sup_{x\in A} {\rm{dist}}(x,B)$, where ${\rm{dist}}(x,B):=\displaystyle\inf_{x'\in B}\|x-x'\|$ stands for the distance from a point $x\in\mathbb{R}^n$ to the set $B$.

\subsection{A Concentration Inequality for Bounded Martingale Differences}
The lemma given below plays a crucial role in proving the uniform exponential convergence of SAA with AMIS.
Actually, it can be thought of as an application of Bennett's Lemma to bounded martingale differences.

\begin{lemma}\label{lem1}
Let $\{\mathscr{Y}_i\}_{i=1}^{\infty}$ be the natural filtration of the real valued random variables $\{Y_i\}_{i=1}^{\infty}$,
and $E[Y_i\mid{\mathscr{Y}_{i-1}}]=0$. In particular, we denote $E[Y_1\mid{\mathscr{Y}_{0}}]\triangleq{E[Y_1]}$.
Suppose that there exist constants $b>0$ and $k>0$ such that $Y_i\leq{b}$ and $E[Y_i^2\mid{\mathscr{Y}_{i-1}}]\leq{k}$. Then, for any $\lambda\geq0$,
\begin{eqnarray}\label{eqn1}
E\left[e^{\lambda{V_n}}\right]\leq\left(\frac{b^{2}e^{-\frac{\lambda{k}}{b}}+ke^{\lambda{b}}}{b^{2}+k}\right)^n,
\end{eqnarray}
where $V_n\triangleq\displaystyle\sum^{n}_{i=1}Y_i$. Moreover, for any $0<\epsilon<b$,
\begin{eqnarray}\label{eqn2}
{\rm{Prob}}\left({\textstyle{\frac{1}{n}}}V_n\geq{\epsilon}\right)\leq{{e}^{-n\mathcal{H}\left(\frac{b\epsilon+k}{b^{2}+k}\big{|}\frac{k}{b^{2}+k}\right)}},
\end{eqnarray}
where $\mathcal{H}(p|p')\triangleq{p\ln\big(\frac{p}{p'}\big)+(1-p)\ln\big(\frac{1-p}{1-p'}\big)}$ for $p, p'\in(0,1)$.
\end{lemma}

$\mathcal{H}$ can be viewed as the relative entropy.
This lemma is close to Corollary 2.4.7 in $\cite{1}$, but is slightly different. For completeness and clarity, the proof is provided here.
\begin{proof}
Since $E[Y_i\mid{\mathscr{Y}_{i-1}}]=0$, $Y_i\leq{b}$ and $E[Y_i^2\mid{\mathscr{Y}_{i-1}}]\leq{k}$ for $i=1,2,...$, applying Bennett's Lemma (see
$\cite[Lemma\ 2.4.1]{1}$), we have
\begin{eqnarray}\label{eqn3}
E\left[e^{\lambda{Y_i}}\mid{\mathscr{Y}_{i-1}}\right]\leq\frac{b^{2}e^{-\frac{\lambda{k}}{b}}+ke^{\lambda{b}}}{b^{2}+k}.
\end{eqnarray}
In particular, when $i=n=1$, we directly get $\eqref{eqn1}$. In other cases, it is obtained by iterating the law of total expectation.
To be specific, combining with $\eqref{eqn3}$, we get
\begin{eqnarray}\label{eqn4}
E\left[e^{\lambda{V_n}}\right]=E\left[e^{\lambda{V_{n-1}}}E\left[e^{\lambda{Y_n}}\mid{\mathscr{Y}_{n-1}}\right]\right]\leq{E\left[e^{\lambda{V_{n-1}}}\right]}\frac{b^{2}e^{-\frac{\lambda{k}}{b}}+ke^{\lambda{b}}}{b^{2}+k}.
\end{eqnarray}
By iteratively applying $\eqref{eqn4}$ until the subscript of ${V_n}$ reaches 1, we obtain the result expressed in $\eqref{eqn1}$.\par
According to Chebycheff's inequality, we have that for any $\epsilon>0$ and $\lambda>0$,
\begin{eqnarray}
{\rm{Prob}}\left({\textstyle{\frac{1}{n}}}V_n\geq{\epsilon}\right)\leq{e^{-\lambda{n\epsilon}}}E\left[e^{\lambda{V_n}}\right].
\nonumber
\end{eqnarray}
Hence, combining with $\eqref{eqn1}$, we obtain
\begin{eqnarray}\label{eqn5}
{\rm{Prob}}\left({\textstyle{\frac{1}{n}}}V_n\geq{\epsilon}\right)\leq{e^{-\lambda{n\epsilon}}}\left(\frac{b^{2}e^{-\frac{\lambda{k}}{b}}+ke^{\lambda{b}}}{b^{2}+k}\right)^n.
\end{eqnarray}
Let $0<\epsilon<b$ and $\lambda=\frac{b}{b^{2}+k}\ln\left(\frac{b(b\epsilon+k)}{k(b-\epsilon)}\right)$.
It's not hard to verify that $\lambda>0$ in this case.
Substituting them into the right-hand side of inequality $\eqref{eqn5}$, we have
\begin{eqnarray}\label{eqn6}
{\rm{Prob}}\left({\textstyle{\frac{1}{n}}}V_n\geq{\epsilon}\right)\leq\left(\bigg(\frac{k}{b\epsilon+k}\bigg)^{\frac{b\epsilon+k}{b^{2}+k}}\bigg(\frac{b}{b-\epsilon}\bigg)^{\frac{b^{2}-b\epsilon}{b^{2}+k}}\right)^n.
\end{eqnarray}
On the other hand, according to the definition of $\mathcal{H}$,
\begin{eqnarray}
{\mathcal{H}}\left(\frac{b\epsilon+k}{b^{2}+k}\,\middle\vert\, \frac{k}{b^{2}+k}\right)={\frac{b\epsilon+k}{b^{2}+k}}\ln\bigg(\frac{b\epsilon+k}{k}\bigg)+{\frac{b^{2}-b\epsilon}{b^{2}+k}}\ln\bigg(\frac{b-\epsilon}{b}\bigg).
\nonumber
\end{eqnarray}
Therefore,
\begin{eqnarray}\label{eqn7}
{{e}^{-n\mathcal{H}\left(\frac{b\epsilon+k}{b^{2}+k}\big{|}\frac{k}{b^{2}+k}\right)}}=\left(\bigg(\frac{b\epsilon+k}{k}\bigg)^{\frac{b\epsilon+k}{b^{2}+k}}\bigg(\frac{b-\epsilon}{b}\bigg)^{\frac{b^{2}-b\epsilon}{b^{2}+k}}\right)^{-n}.
\end{eqnarray}
Thus, combining $\eqref{eqn6}$ and $\eqref{eqn7}$, we have $\eqref{eqn2}$.
\end{proof}

Here, we emphasize that $\mathcal{H}\left(\frac{b\epsilon+k}{b^{2}+k}\big{|}\frac{k}{b^{2}+k}\right)>0$ in $\eqref{eqn2}$ always holds. A detailed analysis is stated in Remark $\ref{rem1}$.

\begin{remark}\label{rem1}
$\mathcal{H}(p|p')>0$ for $0<p'<p<1$. This is easy to verify. By the properties of the logarithm, we have
\begin{eqnarray}
-\mathcal{H}(p|p')&=&p\ln\left(\frac{p'}{p}\right)-(1-p)\ln\left(\frac{1-p}{1-p'}\right)
\nonumber \\
&<&p\left(\frac{p'}{p}-1\right)+(1-p)\left(\frac{1-p'}{1-p}-1\right)
\nonumber \\
&=&0.
\nonumber
\end{eqnarray}
Obviously, $\frac{b\epsilon+k}{b^{2}+k} > \frac{k}{b^{2}+k}$ in Lemma $\ref{lem1}$.
Hence, $\mathcal{H}\left(\frac{b\epsilon+k}{b^{2}+k}\big{|}\frac{k}{b^{2}+k}\right)>0$ in $\eqref{eqn2}$.
\end{remark}

\subsection{Several Important Theorems in Banach spaces}
The results introduced in this subsection are essential for the proof of the main theorem in \cref{sec4}.
In order to state the following conclusions, first we have to make a definition clear, which is the Hadamard directional derivative.

\begin{definition}
Let $B_1$ and $B_2$ be two Banach spaces. A mapping $\mathfrak{F}: B_1\rightarrow{B_2}$ is said to be Hadamard directionally differentiable at $x\in B_1$ in
the direction $\mathbf{d}\in B_1$ if the following limit exists:
\begin{eqnarray}
\mathfrak{F}'_{\mathscr{H}}(x,\mathbf{d}):=\displaystyle\lim_{t\rightarrow0,
~\mathbf{d}'\rightarrow\mathbf{d}}\frac{\mathfrak{F}(x+t\mathbf{d}')-\mathfrak{F}(x)}{t}.
\nonumber
\end{eqnarray}
\end{definition}

The following proposition provides a sufficient condition for the equivalence relation between the directional derivative and the Hadamard directional derivative.
\begin{proposition}\label{pro3}
Let $B_1$ and $B_2$ be two Banach spaces. Consider a mapping $\mathfrak{F}:B_1\rightarrow{B_2}$. Suppose that $\nu\in{B_1}$.
If $\mathfrak{F}(\cdot)$ is directionally differentiable at $\nu$ and Lipschitz continuous in a neighborhood of $\nu$, then $\mathfrak{F}(\cdot)$ is Hadamard directionally differentiable at $\nu$, and $\mathfrak{F}'_{\mathscr{H}}(\nu,\mathbf{d})=\mathfrak{F}'(\nu,\mathbf{d})$, where $\mathbf{d}\in B_1$ is a direction.
\end{proposition}

The definition and proposition above can be referred to $\cite{5,4}$. More concepts and properties about various directional derivatives can be obtained from these references and $\cite[page\ 30]{10}$.\par

The Delta method is a frequently used tool in the asymptotic analysis of stochastic problems. The Delta theorem used here is shown in $\cite[Theorem\ 7.67]{7}$.
\begin{theorem}\label{the3}(Delta Theorem)\par
Let $B_1$ and $B_2$ be two Banach spaces, equipped with their Borel $\sigma$-algebras, where $B_1$ is separable.
Suppose that a mapping $Q: B_1\rightarrow{B_2}$ is Hadamard directionally differentiable at $l\in{B_1}$.
Assume that, as $n\rightarrow\infty$, a sequence of positive numbers $\{\varsigma_n\}$ and a random sequence $\{\mathscr{Z}_n\}$ of $B_1$ satisfy $\varsigma_n\rightarrow\infty$ and $\varsigma_n(\mathscr{Z}_n-l)\stackrel{\mathcal{D}}{\longrightarrow}\mathscr{Z}$, respectively, where $\mathscr{Z}\in{B_1}$.
Then
\begin{eqnarray}
\varsigma_n[Q(\mathscr{Z}_n)-Q(l)]\stackrel{\mathcal{D}}{\longrightarrow}Q'_{\mathscr{H}}(\mathscr{Z},l),~~~~~~~~~~~~~~~~~~~~~~~~
\nonumber  \\
{\rm{and}}~~~~~~~~~~~~~~~~\varsigma_n[Q(\mathscr{Z}_n)-Q(l)]=Q'_{\mathscr{H}}\big(\varsigma_n(\mathscr{Z}_n-l),l\big)+o_p(1).~~~~~~~~~~~~~~~
\nonumber
\end{eqnarray}
\end{theorem}

Danskin's theorem provides sufficient conditions for directional derivability of max-functions, see $\cite[Theorem\ 7.25]{7}$. For the convenience of subsequent use, we directly present the corresponding conclusion for min-functions in Corollary $\ref{cor2}$.
\begin{theorem}\label{the2}(Danskin's Theorem)\par
Let $\Phi:{\mathbb{R}^n}\times\Lambda\rightarrow\mathbb{R}$, where $\Lambda$ is a nonempty and compact topological space.
Suppose that $\Phi(\cdot,\lambda)$ is differentiable for each $\lambda\in\Lambda$ and $\nabla_x\Phi(x,\lambda)$ is continuous
on ${\mathbb{R}^n}\times\Lambda$. Then $\phi(x):=\displaystyle\sup_{\lambda\in\Lambda}\Phi(x,\lambda)$ is locally Lipschitz continuous and directionally differentiable. Furthermore, for a given direction $\mathbf{d}\in{\mathbb{R}^n}$, the directional derivative of $\phi(x)$ is
\begin{eqnarray}
\phi'(x,\mathbf{d})=\displaystyle\sup_{\lambda\in\bar{\Lambda}(x)}\mathbf{d}^\top\nabla_x\Phi(x,\lambda),
\nonumber
\end{eqnarray}
where $\bar{\Lambda}(x):=\arg\displaystyle\max_{\lambda\in\Lambda}\Phi(x,\lambda)$.
\end{theorem}

\begin{corollary}\label{cor2}
Let $\Phi(x,\lambda)$ and $\phi(x)$ be defined as in Theorem $\ref{the2}$. Suppose that $\Psi(x,\lambda)=-\Phi(x,\lambda)$ and $\psi(x)=-\phi(x)$.
Then $\psi(x)=\displaystyle\inf_{\lambda\in\Lambda}\Psi(x,\lambda)$ is locally Lipschitz continuous and directionally differentiable. Furthermore, for a given direction $\mathbf{d}\in{\mathbb{R}^n}$, the directional derivative of $\psi(x)$ is
\begin{eqnarray}
\psi'(x,\mathbf{d})=\displaystyle\inf_{\lambda\in\tilde{\Lambda}(x)}\mathbf{d}^\top\nabla_x\Psi(x,\lambda),
\nonumber
\end{eqnarray}
where $\tilde{\Lambda}(x):=\arg\displaystyle\min_{\lambda\in\Lambda}\Psi(x,\lambda)$.
In particular, if $\tilde{\Lambda}(x_0)=\{\lambda_0\}$ is a singleton for some $x_0\in{\mathbb{R}^n}$, then $\psi(x)$ is differentiable at $x_0$ and $\nabla\psi(x_0)=\nabla_x\Psi(x_0,\lambda_0)$.
\end{corollary}

\subsection{A Central Limit Theorem for Martingale Difference Arrays}
\label{sec2.4}
The crucial step for the proof of the asymptotics is to establish a suitable functional CLT for martingale difference sequences.
In other words, we need a CLT in continuous functional spaces, which works for martingale difference sequences. To this end, we introduce a CLT for martingale difference arrays first, see $\cite[Theorem\ 3.2]{3}$.
It is actually seen as an application of Theorem 5 in $\cite{8}$.
In \cref{sec4}, we transform it into the form we need.

Let $(\mathcal{S}, d)$ be a compact metric space and $C(\mathcal{S})$ denote the space of continuous functions on $\mathcal{S}$, equipped with the sup-norm.
Let $\rho$ be a continuous pseudo-distance (also called continuous semi-metric) w.r.t. $d$ on $\mathcal{S}$.
For $X\in{C(\mathcal{S})}$, denote $q_\rho(X)=\displaystyle\sup_{s_1,s_2\in\mathcal{S}; s_1\neq{s_2}}\frac{|X(s_1)-X(s_2)|}{\rho(s_1,s_2)}$ and
$\|X\|_{C_\rho}=\max\left\{\displaystyle\sup_{s\in\mathcal{S}}|X(s)|, q_\rho(X)\right\}$.
Let $H(\mathcal{S}, \rho, r)$ denote the metric entropy of $(\mathcal{S}, \rho)$ w.r.t. $r$, that is, it's
the logarithm of the minimal number of $\rho$-balls with radius $r$ covering $\mathcal{S}$.

Let $\{Z_{nj}\}_{1\leq{j}\leq{J_n}, n\in{\mathbb{N}}}$ be a triangular array of random elements of the real Banach space $B$, defined on $(\Omega,\mathscr{F},P)$.
Assume further that $\{\mathscr{F}_{nj}\}_{0\leq{j}\leq{J_n}, n\in{\mathbb{N}}}$ is a set of sub-$\sigma$-fields of $\mathscr{F}$ such that $Z_{nj}$ is $\mathscr{F}_{nj}$-measurable with $\mathscr{F}_{nj-1}\subset\mathscr{F}_{nj}$ and $Z_{nj}\in\mathscr{F}_{nj}$.
Then $\{Z_{nj},\mathscr{F}_{nj}\}$ is called a martingale difference array $\cite{3}$ or a martingale triangular array $\cite{6}$ if $E[Z_{nj}\mid{\mathscr{F}_{nj-1}}]=0$ for $1\leq{j}\leq{J_n}$ and every $n\in{\mathbb{N}}$.

\begin{theorem}\label{pro4}
Let $\{Z_{nj},{\mathscr{F}_{nj}}\}$ be a martingale difference array of the space $C(\mathcal{S})$ and satisfy the following conditions:\par
\noindent $\boldsymbol{\rm(C1)}$ There exists a function $\varphi: \mathcal{S}\rightarrow\mathbb{R}_+$ such that for any $s\in\mathcal{S}$
\begin{eqnarray}
~~~~~~~~~~~~~~\displaystyle\sum_{j=1}^{J_n}E[Z^2_{nj}(s)\mid{\mathscr{F}_{nj-1}}]\stackrel{P}{\longrightarrow}\varphi(s)~~~~~~{\rm{as}}~~ n\rightarrow\infty;
\nonumber
\end{eqnarray}

\noindent $\boldsymbol{\rm(C2)}$ For every $s\in\mathcal{S}$ and $\epsilon>0$,
\begin{eqnarray}
~~~~~~\displaystyle\sum_{j=1}^{J_n}E[Z^2_{nj}(s)I_{\left(|Z_{nj}(s)|>\epsilon\right)}\mid{\mathscr{F}_{nj-1}}]\stackrel{P}{\longrightarrow}0~~~~~~{\rm{as}}~~ n\rightarrow\infty;
\nonumber
\end{eqnarray}

\noindent $\boldsymbol{\rm(C3)}$ For a continuous pseudo-distance $\rho$ that satisfies $\displaystyle\int_{0}^{a}{H^{\frac{1}{2}}(\mathcal{S},\rho,r)}{\mathrm{d}{r}}<\infty$ for some $a>0$, we have
\begin{eqnarray}
\displaystyle\sup_{n\in\mathbb{N}}\displaystyle\sum_{j=1}^{J_n}E[\|Z_{nj}\|^2_{C_\rho}]<\infty.
\nonumber
\end{eqnarray}
\end{theorem}
\noindent Then there exists a Gaussian measure $\mu$ such that
\begin{eqnarray}
\displaystyle\sum_{j=1}^{J_n}Z_{nj}\stackrel{\mathcal{D}}{\longrightarrow}\mu  ~~~~~~{\rm{as}}~~ n\rightarrow\infty.
\nonumber
\end{eqnarray}

For a given $s\in\mathcal{S}$, the condition $\boldsymbol{\rm(C2)}$ is also called the Conditioned Lindeberg Condition, see $\cite{6}$.

\section{Uniform Exponential Convergence of SAA with AMIS}
\label{sec3}
Starting from this section, we will adhere to the following settings.

Let $\{\theta_i\}_{i=1}^{\infty}$ be a sequence of random vectors on $(\Omega,\mathscr{F},P)$, where
$\theta_i:\Omega\rightarrow\mathbb{R}^r$ is $\mathscr{F}$-measurable.
Let $\{\mathscr{G}_i\}_{i=1}^{\infty}$ be its corresponding natural filtration.
This implies that the information of $\theta_1,...,\theta_i$ is contained in $\mathscr{G}_i$.
For the sake of ease of presentation, denote $\mathscr{G}_0=\{\emptyset,\Omega\}$.
This implies that $E[\theta_1]=E[\theta_1\mid\mathscr{G}_0]$.
Moreover, for a given $\mathscr{G}_{i-1}$, assume that the conditional distribution of $\theta_i$ has a density $\psi_i$.
The corresponding support of $\psi_i$ is $\Theta_i$, where $\Theta_i\subset\mathbb{R}^r$.

Let $F:\mathcal{X}\times\mathbb{R}^r\rightarrow\mathbb{R}$ be a real valued function.
Moreover, $\mathcal{X}\times\Theta$ contains its support, where ${\mathcal{X}}$ is a compact subset of ${\mathbb{R}^n}$ and $\Theta\subset\mathbb{R}^r$.
Here, we use $\theta$ to denote the random vector $\theta(\omega)$ and a deterministic vector. The particular meaning
will be clear from the context.

\begin{assumption}\label{ass1}
For any $i\in\mathbb{N}$, $\Theta\subset\Theta_i$ a.s.
\end{assumption}

\begin{assumption}\label{ass2}
For any $x\in\mathcal{X}$, $F(x,\cdot)$ is an integrable function, that is, $f(x)=\displaystyle\int_{\Theta}{F(x,\theta)}{\mathrm{d}{\theta}}<\infty$.
\end{assumption}

The above two assumptions are the same as in $\cite{2}$.
For clarity of presentation, we emphasize that $F(\cdot,\theta)=0$ when $\theta\in\Theta_i\backslash\Theta$.

\begin{assumption}\label{ass3}
For a given $x\in\mathcal{X}$, ${\displaystyle\sup_{i\in\mathbb{N}}}{\displaystyle\sup_{\theta_{i}\in\Theta_{i}}}\left|{\frac{F(x,\theta_{i})}{\psi_{i}(\theta_{i})}}\right|<\infty$ a.s.
\end{assumption}

\begin{assumption}\label{ass4}
For any $i\in\mathbb{N}$ and every $x\in\mathcal{X}$, $\frac{F^2(x,\cdot)}{\psi_{i}(\cdot)}$ is integrable on $\Theta_{i}$.
\end{assumption}

\begin{lemma}\label{lem2}
Suppose Assumption $\ref{ass1}$ and $\ref{ass2}$ hold. Let $\Upsilon_i(x)=\frac{F(x,\theta_{i})}{\psi_{i}(\theta_{i})}-f(x)$ and $S_n(x)=\displaystyle\sum^{n}_{i=1}\Upsilon_i(x)$, where $n\in\mathbb{N}$. Then for a fixed $x\in{\mathcal{X}}$, $\{\Upsilon_i(x), \mathscr{G}_i\}$ is a martingale difference sequence and $\{S_n(x),\mathscr{G}_n\}$ is a martingale.
\end{lemma}

This property is used in $\cite[Theorem\ 6]{2}$ for the proof of pointwise convergence.
For the convenience of reading this paper, we give a brief  proof.
\begin{proof}
For any $i\in\mathbb{N}$, we know that the conditional distribution of $\theta_i$ has a density $\psi_i$ under $\mathscr{G}_{i-1}$.
Hence, for a given $x\in{\mathcal{X}}$,
\begin{eqnarray}
E[\Upsilon_i(x)\mid{\mathscr{G}_{i-1}}]=\displaystyle\int_{\Theta_i}{\Upsilon_i(x)}{\psi_{i}(\theta_{i})}{\mathrm{d}{\theta_{i}}}
=\displaystyle\int_{\Theta_i}\left[{F(x,\theta_{i})}-f(x){\psi_{i}(\theta_{i})}\right]{\mathrm{d}{\theta_{i}}}=0.
\nonumber
\end{eqnarray}
This means that $\{\Upsilon_i(x), \mathscr{G}_i\}$ is a martingale difference sequence.
It's obvious that $\{S_n(x),\mathscr{G}_n\}$ is a martingale.
\end{proof}

Henceforth, $\Upsilon_i(x)$ and $S_n(x)$ in the paper are defined as in Lemma $\ref{lem2}$.
\begin{lemma}\label{lem3}
Suppose that Assumptions $\ref{ass1}$-$\ref{ass4}$ hold. Then for all $x\in{\mathcal{X}}$, there exist $L(x)>0$ and $C(x)>0$ such that $|\Upsilon_i(x)|\leq{L(x)}$ and
$E[\Upsilon^2_i(x)\mid{\mathscr{G}_{i-1}}]\leq{C(x)}$.
\end{lemma}

\begin{proof}

By Assumption $\ref{ass4}$, $E[\Upsilon^2_i(x)\mid{\mathscr{G}_{i-1}}]$ is well defined for each $x$.
Furthermore, according to Assumption $\ref{ass3}$, there is a function $\beta(x)>0$ such that for any $i\in\mathbb{N}$ and all $x\in\mathcal{X}$, $\left|\frac{F(x,\theta_{i})}{\psi_{i}(\theta_{i})}\right|\leq\beta(x)$.
Therefore, for a fixed $x$, $|\Upsilon_i(x)|=\left|\frac{F(x,\theta_{i})}{\psi_{i}(\theta_{i})}-f(x)\right|\leq{\beta(x)+|f(x)|}$. From the above analysis, we have
\begin{eqnarray}
E[\Upsilon^2_i(x)\mid{\mathscr{G}_{i-1}}]&=&\displaystyle\int_{\Theta_i}{\Upsilon^2_i(x)}{\psi_{i}(\theta_{i})}{\mathrm{d}{\theta_{i}}}
\nonumber \\
&=&\displaystyle\int_{\Theta_i}{\left[{\frac{F(x,\theta_{i})}{\psi_{i}(\theta_{i})}-f(x)}\right]^2}{\psi_{i}(\theta_{i})}{\mathrm{d}{\theta_{i}}}
\nonumber \\
&\leq&\displaystyle\int_{\Theta_i}\left[\beta(x){F(x,\theta_{i})}-2f(x){F(x,\theta_{i})}+f^2(x){\psi_{i}(\theta_{i})}\right]{\mathrm{d}{\theta_{i}}}
\nonumber \\
&=&\beta(x){f(x)}-f^2(x)
\nonumber \\
&\leq&\beta(x){(|f(x)|+1)}.
\nonumber
\end{eqnarray}
Denote $L(x)={\beta(x)+|f(x)|}$ and $C(x)=\beta(x){(|f(x)|+1)}$.
The proof is complete.
\end{proof}

Combining Lemma $\ref{lem2}$ and Lemma $\ref{lem3}$, applying the concentration inequality in Lemma $\ref{lem1}$, we obtain pointwise weak laws of large numbers (P-WLLN) as follows.

\begin{proposition}\label{pro1}
Suppose Assumptions $\ref{ass1}$-$\ref{ass4}$ hold. Then for any given $x\in{\mathcal{X}}$,
there exist $b(x)>0$ and $k(x)>0$ such that for $0<\epsilon<b(x)$,
\begin{eqnarray}
{\rm{Prob}}\left({\textstyle{\frac{1}{n}}}S_n(x)\geq{\epsilon}\right)\leq{{e}^{-n\mathcal{H}\left(\frac{b(x)\epsilon+k(x)}{b^{2}(x)+k(x)}\big{|}\frac{k(x)}{b^{2}(x)+k(x)}\right)}},\label{eqn8}
\\
{\rm{Prob}}\left({\textstyle{\frac{1}{n}}}S_n(x)\leq{-\epsilon}\right)\leq{{e}^{-n\mathcal{H}\left(\frac{b(x)\epsilon+k(x)}{b^{2}(x)+k(x)}\big{|}\frac{k(x)}{b^{2}(x)+k(x)}\right)}}.\label{eqn14}
\end{eqnarray}
\end{proposition}

\begin{proof}
For a fixed $x\in{\mathcal{X}}$, it follows from Lemmas $\ref{lem2}$ and Lemma $\ref{lem3}$ that
\begin{eqnarray}
E[\Upsilon_i(x)\mid{\mathscr{G}_{i-1}}]=0,~~ |\Upsilon_i(x)|\leq{b(x)} ~~{\rm{and}}~~
E[\Upsilon^2_i(x)\mid{\mathscr{G}_{i-1}}]\leq{k(x)}
\nonumber
\end{eqnarray}
with $b(x)>0$ and $k(x)>0$.
This implies that $\Upsilon_i(x)$ satisfies the conditions in Lemma $\ref{lem1}$.
Therefore, for $0<\epsilon<b(x)$, $(\ref{eqn8})$ holds.\par
On the other hand, let $\Gamma_i(x)=-\Upsilon_i(x)$, then it is not difficult to verify that
for $b(x)>0$ and $k(x)>0$ above,
\begin{eqnarray}
&&E[\Gamma_i(x)\mid{\mathscr{G}_{i-1}}]=-E[\Upsilon_i(x)\mid{\mathscr{G}_{i-1}}]=0,
\nonumber \\
&&~~~~~~~~~~~|\Gamma_i(x)|=|\Upsilon_i(x)|\leq{b(x)},
\nonumber \\
&&E[\Gamma^2_i(x)\mid{\mathscr{G}_{i-1}}]=E[\Upsilon^2_i(x)\mid{\mathscr{G}_{i-1}}]\leq{k(x)}.
\nonumber
\end{eqnarray}
Therefore, using Lemma $\ref{lem1}$ again, we have
\begin{eqnarray}
{\rm{Prob}}\left({\textstyle{-\frac{1}{n}}}S_n(x)\geq{\epsilon}\right)\leq{{e}^{-n\mathcal{H}\left(\frac{b(x)\epsilon+k(x)}{b^{2}(x)+k(x)}\big{|}\frac{k(x)}{b^{2}(x)+k(x)}\right)}},
\nonumber
\end{eqnarray}
where $0<\epsilon<b(x)$.
Thus, $(\ref{eqn14})$ is proved.
\end{proof}

According to the work of Andrews, it is shown that the weak uniform convergence can be obtained by P-WLLN combined with the weak Lipschitz (W-LIP) assumption if ${\mathcal{X}}$ is totally bounded, see $\cite[Theorem\ 3(a)]{11}$.
Therefore, in order to obtain uniform exponential convergence, we need some additional conditions.
Inspired by W-LIP, we need the help of the following definition.
Its form draws on Xu's extension of the definition of calmness, see reference $\cite{12}$ for details.
A basic definition of calmness can be found  in $\cite{13}$.
\begin{definition}\label{def1}(Adaptive Multiple Modulus Calmness {\rm\textbf{(AMMC for short)}}). 
Let $g:\textbf{X}\times\Theta\rightarrow\mathbb{R}$ be a real valued function, where $\textbf{X}$ is a closed subset of ${\mathbb{R}^n}$ and $\Theta\subset\mathbb{R}^r$.
Let ${\theta}:\Omega\rightarrow\Theta$ be a random vector defined on a probability space $(\Omega,\mathscr{F},P)$.
Let $\{\theta_i\}_{i=1}^{\infty}$ and $\{\mathscr{G}_i\}_{i=1}^{\infty}$ be its samples and corresponding natural filtration, respectively.
Suppose that $\varpi(\cdot)$ is a modulus of continuity and Assumption $\ref{ass1}$ holds. Then for a given $x_0\in{\textbf{X}}$, in the sense of the function $\varpi$, $g(x,\theta)$ is said to be\par
\noindent (i) {\textbf{AMMC from below}} at $x_0$ with adaptive multiple coefficient ${\alpha_{i}}$,
if there exist $\delta(x_0)>0$ and a family of integrable functions
${\alpha_{i}}:{\Theta_{i}}\rightarrow{\mathbb{R}_{+}}$ for every ${i}\in{\mathbb{N}}$, which are adapted to the filtration $\mathscr{G}_{i-1}$,  such that
\begin{eqnarray}
g(x,\theta_{i})-g(x_0,\theta_{i})\geq-{\alpha_{i}(\theta_{i})}\varpi(\|x-x_0\|),~~~~~~\forall{\theta_{i}\in\Theta_{i}},
\nonumber
\end{eqnarray}
for all $x\in{\textbf{X}}$ with $0<\|x-x_0\|<\delta(x_0)$;\par
\noindent (ii) \textbf{AMMC from above} at $x_0$ with adaptive multiple coefficient ${\alpha_{i}}$,
if $-g(x,\theta)$ is AMMC from below at $x_0$ with the same adaptive multiple coefficient, i.e.,
\begin{eqnarray}
g(x,\theta_{i})-g(x_0,\theta_{i})\leq{\alpha_{i}(\theta_{i})}\varpi(\|x-x_0\|),~~~~~~~\forall{\theta_{i}\in\Theta_{i}},
\nonumber
\end{eqnarray}
for all $x\in{\textbf{X}}$ with $0<\|x-x_0\|<\delta(x_0)$;\par
\noindent (iii) \textbf{AMMC} at $x_0$ with adaptive multiple coefficient ${\alpha_{i}}$,
if $g(x,\theta)$ is AMMC both from below and above at $x_0$ with the same adaptive multiple coefficient, i.e.,
\begin{eqnarray}
|g(x,\theta_{i})-g(x_0,\theta_{i})|\leq{\alpha_{i}(\theta_{i})}\varpi(\|x-x_0\|),~~~~~~\forall{\theta_{i}\in\Theta_{i}},
\nonumber
\end{eqnarray}
for all $x\in{\textbf{X}}$ with $0<\|x-x_0\|<\delta(x_0)$.
\end{definition}

In particular, AMMC also has the following statements. \par
\noindent(i) If $x_0$ in the inequalities of Definition $\ref{def1}$ is replaced by any $x'$, where $x,x'\in{\textbf{X}}$ and $0<\|x-x'\|<\delta(x_0)$, then $g(x,\theta)$ is called strictly AMMC (/from below/from above) at $x_0$.\par
\noindent(ii) If every point on ${\textbf{X}}$ satisfies Definition $\ref{def1}$, we say that $g(x,\theta)$ is AMMC (/from below/from above) on ${\textbf{X}}$.\par

\begin{lemma}\label{lem4}
Let $g(x,\theta)$ be strictly AMMC (/from below/from above) on $\textbf{X}$ in the sense of the function $\varpi$ with adaptive multiple coefficient ${\alpha_{i}}$, where $\{\theta_i\}_{i=1}^{\infty}$ and $\{\mathscr{G}_i\}_{i=1}^{\infty}$ are its samples and corresponding natural filtration, respectively. If $\textbf{X}$ is a compact set, then there exists a uniformly positive number $\delta$, such that
\begin{eqnarray}
|g(x_1,\theta_i)-g(x_2,\theta_i)|\leq{\alpha_{i}(\theta_{i})}\varpi(\|x_1-x_2\|),~~~~~~\forall{\theta_{i}\in\Theta_{i}},
 \nonumber
\end{eqnarray}
$(~or~g(x_1,\theta_i)-g(x_2,\theta_i)\geq{-\alpha_{i}(\theta_{i})}\varpi(\|x_1-x_2\|$
$
~~or~~g(x_1,\theta_i)-g(x_2,\theta_i)\leq{\alpha_{i}(\theta_{i})}\varpi(\|x_1-x_2\|)
$
for $0<\|{x_1}-{x_2}\|<\delta$ and $x_1,x_2\in\textbf{X}$.
\end{lemma}

\begin{proof}
Since $g(x,\theta)$ is strictly AMMC on $\textbf{X}$, around each $y\in\textbf{X}$, there exists a ball of some radius $\delta_y>0$ such that for any two points $y_1, y_2\in\textbf{X}$ in this ball,
\begin{eqnarray}
~~~~~|g(y_1,\theta_i)-g(y_2,\theta_i)|\leq{\alpha_{i}(\theta_{i})}\varpi(\|y_1-y_2\|),~~~~\forall{i}\in{\mathbb{N}}~~\rm{and}~~\forall{\theta_{i}\in\Theta_{i}}.
 \nonumber
\end{eqnarray}
Then ${\displaystyle\bigcup_{y\in\textbf{X}}}~\mathcal{B}(y,\frac{\delta_y}{2})$ covers $\textbf{X}$. By the compactness of $\textbf{X}$, it follows from the finite covering theorem that a finite number of these balls suffice to cover it.
Thus, there exist points $y_k$ with corresponding radius $\delta_k>0$ for $k=1,...,r$ such that $\textbf{X}\subset{\displaystyle\bigcup^{r}_{k=1}}~\mathcal{B}(y_k,\frac{\delta_k}{2})$.\par
Let $0<\delta_0<\frac{\delta_k}{2}$ for all $1\leq k\leq r$.
For any pair $z_1,z_2\in{\textbf{X}}$ with $0<\|z_1-z_2\|<\delta_0$, there always exists some index $k'\in\{1,...,r\}$ such that $z=\frac{z_1+z_2}{2}\in\mathcal{B}(y_{k'},\frac{\delta_{k'}}{2})$.
Therefore, $z_1, z_2\in\mathcal{B}(z,\delta_0)\subset\mathcal{B}(y_{k'},\delta_{k'})$.
This implies that
\begin{eqnarray}
|g(z_1,\theta_i)-g(z_2,\theta_i)|\leq{\alpha_{i}(\theta_{i})}\varpi(\|z_1-z_2\|),~~~~~~\forall{\theta_{i}\in\Theta_{i}},
 \nonumber
\end{eqnarray}
for any $z_1,z_2\in\textbf{X}$ with $0<\|z_1-z_2\|<\delta_0$ and all ${i}\in{\mathbb{N}}$.
Thus, the proof is complete.
\end{proof}

Now, we state the main theorem of this section.
To simplify notation, denote
\begin{eqnarray}
H_n(x):=f_n(x)-f(x)=\frac{1}{n}\displaystyle\sum_{i=1}^{n}\Upsilon_i(x)=\frac{1}{n}S_n(x).
\nonumber
\end{eqnarray}

The purpose of using W-LIP is to deduce that the sequence studied is stochastically equicontinuous, see $\cite[Lemma\ 2(a)] {11}$.
In the following theorem, the condition $\boldsymbol{\rm(H0)}$ and strictly AMMC ensure that $\{H_n(x)\}$ is stochastically equicontinuous on ${\mathcal{X}}$.
Now, we prove the main theorem in this subsection, which gives uniform exponential convergence of SAA with AMIS.
\begin{theorem}\label{the1}
Suppose that Assumptions $\ref{ass1}$-$\ref{ass4}$ hold.
Let $f(x)$ be continuous on the compact set ${\mathcal{X}}$.
Then the following statements hold:\par
\noindent (i) If $F(x,\theta)$ is strictly AMMC from below on ${\mathcal{X}}$ in the sense of the function $\varpi$ with adaptive multiple coefficient ${\alpha_{i}}$, and \par
\noindent $\boldsymbol{\rm(H0)}$ for any $\varepsilon>0$, there exist constants $M(\varepsilon)>0$ and $\tau(\varepsilon)>0$ such that
\begin{eqnarray}\label{eqn12}
{\rm{Prob}}\big(A_n\geq\varepsilon{M(\varepsilon)}\big)\leq{e^{-n\tau(\varepsilon)}},
\end{eqnarray}
where $A_n:=\frac{1}{n}\displaystyle\sum_{i=1}^{n}A_i(\theta_i)$ and $A_i(\theta_i):=\frac{\alpha_i(\theta_{i})}{\psi_{i}(\theta_{i})}$, \par
\noindent then for every $\varepsilon>0$, there exist positive constants $\lambda(\varepsilon)$ and $\gamma(\varepsilon)$, such that
\begin{eqnarray}
{\rm{Prob}}\left\{\displaystyle\inf_{x\in{{\mathcal{X}}}}H_n(x)\leq{-\varepsilon}\right\}\leq\lambda(\varepsilon)e^{-n\gamma(\varepsilon)};
\nonumber
\end{eqnarray}
\noindent (ii) If $F(x,\theta)$ is strictly AMMC from above on ${\mathcal{X}}$ in the sense of the function $\varpi$ with adaptive multiple coefficient ${\alpha_{i}}$, and
\noindent $\boldsymbol{\rm(H0)}$ holds, then for every $\varepsilon>0$, there exist positive constants $\lambda(\varepsilon)$ and $\gamma(\varepsilon)$ such that
\begin{eqnarray}
{\rm{Prob}}\left\{\displaystyle\sup_{x\in{\mathcal{X}}}H_n(x)\geq{\varepsilon}\right\}\leq\lambda(\varepsilon)e^{-n\gamma(\varepsilon)};
\nonumber
\end{eqnarray}
\noindent (iii) If $F(x,\theta)$ is strictly AMMC on ${\mathcal{X}}$ in the sense of the function $\varpi$ with adaptive multiple coefficient ${\alpha_{i}}$, and
\noindent $\boldsymbol{\rm(H0)}$ holds, then for every $\varepsilon>0$, there exist positive constants $\lambda(\varepsilon)$ and $\gamma(\varepsilon)$ such that
\begin{eqnarray}
{\rm{Prob}}\left\{\displaystyle\sup_{x\in{\mathcal{X}}}|H_n(x)|\geq{\varepsilon}\right\}\leq\lambda(\varepsilon)e^{-n\gamma(\varepsilon)}.
\nonumber
\end{eqnarray}

\end{theorem}

\begin{proof}
(i) For any given $\varepsilon>0$ and fixed $x\in{\mathcal{X}}$, it follows from Proposition $\ref{pro1}$ that
there exist $b(x)>0$ and $k(x)>0$ such that for $0<\varepsilon<b(x)$,
\begin{eqnarray}\label{eqn13}
{\rm{Prob}}\left(H_n(x)\leq{-\varepsilon}\right)\leq{{e}^{-n\mathcal{H}\left(\frac{b(x)\varepsilon+k(x)}{b^{2}(x)+k(x)}\big{|}\frac{k(x)}{b^{2}(x)+k(x)}\right)}}.
\end{eqnarray}
For the sake of discussion, we assume that $\displaystyle\inf_{x\in{\mathcal{X}}}{b(x)}=b_0>0$. Therefore, let $0<\varepsilon<b_0$, then $\eqref{eqn13}$ holds for all $x\in\mathcal{X}$.\par

Note that $f(x)$ is a continuous function on the compact set ${\mathcal{X}}$, which implies that $f(x)$ is uniformly continuous. That is, there exists $\delta_1>0$ such that
\begin{eqnarray}\label{eqn16}
|f(x)-f(x')|\leq\frac{\varepsilon}{4}
\end{eqnarray}
for all $x,x'\in{\mathcal{X}}$ with $\|x-x'\|\leq{\delta_1}$.
It follows from Lemma $\ref{lem4}$ that
for $\frac{\varepsilon}{4}>0$ and all $i\in\mathbb{N}$, there exists $\delta_0>0$ such that
\begin{eqnarray}\label{eqn17}
F(x,\theta_i)-F(x',\theta_i)\geq{-\alpha_{i}(\theta_{i})}\varpi(\|x-x'\|)
\end{eqnarray}
for $0<\|x'-x_0\|<\delta_0$.

According to $\boldsymbol{\rm(H0)}$, for $\varepsilon$ given above, there exist constants $M(\frac{\varepsilon}{4})>0$ and $\tau(\frac{\varepsilon}{4})>0$ such that $\eqref{eqn12}$ holds.
Choose a sufficiently small positive constant $\delta<\min\{\delta_0,\delta_1\}$ such that $\varpi(\delta)\leq\frac{1}{M(\frac{\varepsilon}{4})}$. By the definition of modulus of continuity, we have
\begin{eqnarray}\label{eqn18}
\varpi(d)\leq\varpi(\delta)\leq\frac{1}{M(\frac{\varepsilon}{4})},
\end{eqnarray}
for $0<d<\delta$.
We construct a family of sequential balls $\mathfrak{B}_j=\{x:\|x-x_{j}\|<\delta\}$ with center at $x_j\in{\mathcal{X}}$ and radius $\delta$.
Through the finite covering theorem, ${\mathcal{X}}$ can be covered by only $K$ balls, by choosing appropriate $\{x_j\}$, where $j\in\{1,\cdots,K\}\triangleq{J}$.\par
Since for every $j\in{J}$,
$H_n(x)=H_n(x_j)+H_n(x)-H_n(x_j)$,
we obtain 
\begin{eqnarray}
\displaystyle\inf_{x\in{\mathcal{X}}}H_n(x)\geq\displaystyle\min_{j\in{J}}H_n(x_j)+\displaystyle\inf_{{j\in{J}},x\in{\mathfrak{B}_j}}\big(H_n(x)-H_n(x_j)\big).
\nonumber
\end{eqnarray}
Hence
\begin{eqnarray}\label{eqn19}
&{\rm{Prob}}\left\{\displaystyle\inf_{x\in{{\mathcal{X}}}}H_n(x)\leq{-\varepsilon}\right\}
\leq{\rm{Prob}}\left\{\displaystyle\min_{j\in{J}}H_n(x_j)\leq-\frac{\varepsilon}{2}\right\}~~~~~~~~~~~~~~~~~~~~~~~~~~~~~
\nonumber \\
&~~~~~~~~~~~~~~~~~~~~~~~~~~~~~~~+{\rm{Prob}}\left\{\displaystyle\inf_{{j\in{J}}, x\in{\mathfrak{B}_j}}\big(H_n(x)-H_n(x_j)\big)\leq-\frac{\varepsilon}{2}\right\}.
\end{eqnarray}
According to ($\ref{eqn13}$), we get
\begin{eqnarray}\label{eqn15}
{\rm{Prob}}\left\{\displaystyle\min_{j\in{J}}H_n(x_j)\leq{-\frac{\varepsilon}{2}}\right\}&\leq&\displaystyle\sum_{j=1}^{K}{\rm{Prob}}\left\{{H_n(x_{j})}\leq{-\frac{\varepsilon}{2}}\right\}
\nonumber\\
&\leq&\displaystyle\sum_{j=1}^{K}{e^{-n\mathcal{H}\left(\frac{\frac{\varepsilon}{2}b(x_j)+k(x_j)}{b^{2}(x_j)+k(x_j)}\big{|}\frac{k(x_j)}{b^{2}(x_j)+k(x_j)}\right)}} .
\end{eqnarray}
For any $j\in{J}$,
\begin{eqnarray}
H_n(x)-H_n(x_j)&=&\frac{1}{n}\displaystyle\sum_{i=1}^{n}[\Upsilon_i(x)-\Upsilon_i(x_j)]
\nonumber \\
&=&\frac{1}{n}\displaystyle\sum_{i=1}^{n}\left[\frac{F(x,\theta_{i})}{\psi_{i}(\theta_{i})}-\frac{F(x_j,\theta_{i})}{\psi_{i}(\theta_{i})}\right]+f(x_j)-f(x),
\nonumber
\end{eqnarray}
where $x\in{\mathfrak{B}_j}$ and $x\neq{x_j}$.
By $\eqref{eqn16}$, $f(x_j)-f(x)\geq-\frac{\varepsilon}{4}$. Furthermore, combining $\eqref{eqn17}$ and $\eqref{eqn18}$, we have
\begin{eqnarray}
H_n(x)-H_n(x_j)&\geq&-\frac{1}{n}\displaystyle\sum_{i=1}^{n}\frac{\alpha_i(\theta_{i})}{\psi_{i}(\theta_{i})}\varpi(\|x-x_j\|)-\frac{\varepsilon}{4}
\nonumber \\
&\geq&-\frac{A_n}{M(\frac{\varepsilon}{4})}-\frac{\varepsilon}{4}.
\nonumber
\end{eqnarray}
In particular, when $x=x_j$, $H_n(x)-H_n(x_j)=0\geq-\frac{A_n}{M(\frac{\varepsilon}{4})}-\frac{\varepsilon}{4}$.
Therefore,
\begin{eqnarray}\label{eqn20}
{\rm{Prob}}\left\{\displaystyle\inf_{{j\in{J}},x\in{\mathfrak{B}_j}}\big(H_n(x)-H_n(x_j)\big)\leq-\frac{\varepsilon}{2}\right\}&\leq&{\rm{Prob}}\left\{-\frac{A_n}{M(\frac{\varepsilon}{4})}\leq-\frac{\varepsilon}{4}\right\}
\nonumber \\
&\leq&{e^{-n\tau(\frac{\varepsilon}{4})}}.
\end{eqnarray}
Combining  $\eqref{eqn19}$,  $\eqref{eqn15}$ and  $\eqref{eqn20}$, we get
\begin{eqnarray}
{\rm{Prob}}\left\{\displaystyle\inf_{x\in{{\mathcal{X}}}}H_n(x)\leq{-\varepsilon}\right\}\leq
\displaystyle\sum_{j=1}^{K}{e^{-n\mathcal{H}\left(\frac{\frac{\varepsilon}{2}b(x_j)+k(x_j)}{b^{2}(x_j)+k(x_j)}\big{|}\frac{k(x_j)}{b^{2}(x_j)+k(x_j)}\right)}}+{e^{-n\tau(\frac{\varepsilon}{4})}}.
\nonumber
\end{eqnarray}

The proof of part (i) is complete, while part (ii) is handled analogously. Part (iii) is the combination of part (i) and part (ii).
\end{proof}

Together with the concentration inequality stated in Lemma $\ref{lem1}$, we provide the following sufficient conditions for the condition $\boldsymbol{\rm(H0)}$.
Unless otherwise specified, $A_n$ and $A_i(\theta_i)$ are defined as in Theorem $\ref{the1}$.
\begin{proposition}\label{pro2}
 For each ${i}\in{\mathbb{N}}$, let $\alpha_i(\theta_i)$ and $A_i(\theta_i)$ in Theorem $\ref{the1}$ satisfy the following conditions: \par
\noindent $\boldsymbol{\rm{(H1)}}$ ${\alpha_i(\theta_{i})}$ is integrable on $\Theta_{i}$;
\par
\noindent $\boldsymbol{\rm{(H2)}}$ $\frac{\alpha^{2}_i(\theta_{i})}{\psi_{i}(\theta_{i})}$ is integrable on $\Theta_{i}$;
\par
\noindent $\boldsymbol{\rm{(H3)}}$ $A_i(\theta_i)$ is uniformly upper bounded w.r.t $i$ and $\theta_{i}\in\Theta_{i}$. \par
\noindent Then, $\eqref{eqn12}$ holds.
\end{proposition}

\begin{proof}
This proof is similar to that of Proposition $\ref{pro1}$. For clarity, we show it briefly. According to $\boldsymbol{\rm{(H1)}}$, for each ${i}\in{\mathbb{N}}$,
\begin{eqnarray}
\displaystyle\int_{\Theta_i}{\alpha_i(\theta_{i})}{\mathrm{d}{\theta_{i}}}=
\displaystyle\int_{\Theta_i}{A_i(\theta_{i})}{\psi_{i}(\theta_{i})}{\mathrm{d}{\theta_{i}}}
\triangleq\overline{\alpha}_i>0.
\nonumber
\end{eqnarray}
Hence, $E[A_i(\theta_{i})\mid{\mathscr{G}_{i-1}}]=\overline{\alpha}_i$.\par
According to $\boldsymbol{\rm{(H3)}}$, there exists a positive constant $\hat{k}>0$ such that
\begin{eqnarray}
{\displaystyle\sup_{{i\in\mathbb{N}},{\theta_{i}\in\Theta_{i}}}}{A_{i}(\theta_{i})}\leq{\hat{k}}.
\nonumber
\end{eqnarray}
Therefore, $\overline{\alpha}_i\leq\hat{k}$.
Let $B_i(\theta_i)=A_i(\theta_i)-\overline{\alpha}_i$ and
$B_n=\frac{1}{n}\displaystyle\sum_{i=1}^{n}B_i(\theta_i)$.
Obviously, it's easy to get that $E[B_i(\theta_{i})\mid{\mathscr{G}_{i-1}}]=0$ and $B_i(\theta_{i})\leq{\hat{k}-\overline{\alpha}_i}$.
Furthermore, by $\boldsymbol{\rm{(H2)}}$,
\begin{eqnarray}
E[B^2_i(\theta_{i})\mid{\mathscr{G}_{i-1}}]&=&\displaystyle\int_{\Theta_i}\big({A_i(\theta_{i})-\overline{\alpha}}\big)^{2}{\psi_{i}(\theta_{i})}{\mathrm{d}{\theta_{i}}}
\nonumber \\
&=&\displaystyle\int_{\Theta_i}{\frac{\alpha^{2}_i(\theta_{i})}{\psi_{i}(\theta_{i})}}{\mathrm{d}{\theta_{i}}}-2\overline{\alpha}_i\displaystyle\int_{\Theta_i}{\alpha_{i}(\theta_{i})}{\mathrm{d}{\theta_{i}}}+{\overline{\alpha}_i}^2
\nonumber \\
&=&\displaystyle\int_{\Theta_i}{A_i(\theta_{i})}{\alpha_{i}(\theta_{i})}{\mathrm{d}{\theta_{i}}}-{\overline{\alpha}_i}^2
\nonumber \\
&\leq&\displaystyle\int_{\Theta_i}\hat{k}{\alpha_{i}(\theta_{i})}{\mathrm{d}{\theta_{i}}}-{\overline{\alpha}_i}^2
\nonumber \\
&=&{\overline{\alpha}_i}(\hat{k}-{\overline{\alpha}_i})
\nonumber \\
&\leq&\hat{k}^2\triangleq{\hat{b}}.
\nonumber
\end{eqnarray}
It follows from Lemma $\ref{lem1}$ that for any $0<\xi<\hat{b}$,
\begin{eqnarray}
{\rm{Prob}}\left(B_n\geq{\xi}\right)\leq{{e}^{-n\mathcal{H}\left(\frac{\hat{b}\xi+\hat{k}}{\hat{b}^{2}+\hat{k}}\big{|}\frac{\hat{k}}{\hat{b}^{2}+\hat{k}}\right)}}.
\nonumber
\end{eqnarray}
On the other hand,
\begin{eqnarray}
A_n=B_n+\frac{1}{n}\displaystyle\sum_{i=1}^{n}\overline{\alpha}_i\leq B_n+\hat{k}.
\nonumber
\end{eqnarray}
Therefore, ${\rm{Prob}}\left(A_n\geq{\xi+\hat{k}}\right)\leq{\rm{Prob}}\left(B_n\geq{\xi}\right)$.
In fact, for any $\epsilon>0$, there exists $M(\epsilon)>0$ such that ${\epsilon}M(\epsilon)=\xi+{\hat{k}}$.
Hence
\begin{eqnarray}
{\rm{Prob}}\left(A_n\geq{\epsilon}M(\epsilon)\right)\leq{\rm{Prob}}\left(B_n\geq{{\epsilon}M(\epsilon)-\hat{k}}\right)\leq{{e}^{-n\mathcal{H}\left(\frac{\hat{b}\left({\epsilon}M(\epsilon)-{\hat{k}}\right)+\hat{k}}{\hat{b}^{2}+\hat{k}}\big{|}\frac{\hat{k}}{\hat{b}^{2}+\hat{k}}\right)}}.
\nonumber
\end{eqnarray}
The proof is complete.
\end{proof}

\section{Asymptotics of the Optimal Value for SAA with AMIS}
\label{sec4}

\subsection{Functional CLT for Martingale Difference Sequences}
\label{sec4.1}
In order to prove the asymptotics of the optimal value for SAA with AMIS, we need to use a functional CLT for martingale difference sequences.
It is actually an application of Theorem $\ref{pro4}$, as shown in the following.

\begin{theorem}\label{cor1}
Let $\{X_i, \mathscr{F}_i\}_{i=1}^{\infty}$ be a martingale difference sequence of the space $C(\mathcal{S})$ (defined in \cref{sec2.4}).
Suppose that the following assumptions hold.\par

\noindent $\boldsymbol{\rm(A1)}$
There exists a real nonnegative random sequence $\{M_i\}$ on $(\Omega,\mathscr{F},P)$ and a function $\beta:\mathcal{S}\rightarrow\mathbb{R}\backslash\{0\}$ such that for any $s_1,s_2\in\mathcal{S}$ and all $i\in\mathbb{N}$,
\begin{eqnarray}\label{eqn21}
|Y_i(s_1)-Y_i(s_2)|\leq{M_i}~~{\rm{a.s.}},
\end{eqnarray}
where $Y_i(s):=|\beta(s)X_i(s)|$ and $\displaystyle\sup_{s\in\mathcal{S}}|\beta^{-1}(s)|<\infty$.

\noindent $\boldsymbol{\rm(A2)}$ $\frac{1}{n}\displaystyle\sum_{i=1}^{n}E[M^2_i\mid{\mathscr{F}_{i-1}}]\stackrel{P}{\longrightarrow}0$.

\noindent $\boldsymbol{\rm(A3)}$
For any $i\in\mathbb{N}$, there exists a constant $b>0$ such that $E[M^2_i\mid{\mathscr{F}_{i-1}}]\leq b$ a.s.

\noindent $\boldsymbol{\rm(A4)}$
For some $s_0\in\mathcal{S}$ and all $i\in\mathbb{N}$,
there exists a constant $k>0$ such that $|X_i(s_0)|\leq k$ a.s.

\noindent $\boldsymbol{\rm(A5)}$
$\frac{1}{n}\displaystyle\sum_{i=1}^{n}E[X^2_i(s_0)\mid{\mathscr{F}_{i-1}}]\stackrel{P}{\longrightarrow}c$, where $c$ is a positive constant.

\noindent $\boldsymbol{\rm(A6)}$
There exists a real nonnegative random sequence $\{\varsigma_i\}$ on $(\Omega,\mathscr{F},P)$ and a continuous distance $\rho$ with $\displaystyle\int_{0}^{1}{H^{\frac{1}{2}}(\mathcal{S},\rho,r)}{\mathrm{d}{r}}<\infty$ such that
$\displaystyle\sup_{n\in\mathbb{N}}\frac{1}{n}\displaystyle\sum_{i=1}^{n}E[\varsigma^2_i]<\infty$
and for any $i\in\mathbb{N}$,
\begin{eqnarray}\label{eqn46}
|X_i(s_1)-X_i(s_2)|\leq\varsigma_i{\rho(s_1,s_2)}~~{\rm{a.s.}}
\end{eqnarray}

\noindent Then there exists a Gaussian measure $\mu$ on $C(\mathcal{S})$ such that
\begin{eqnarray}\label{eqn47}
\frac{1}{\sqrt{n}}\displaystyle\sum_{i=1}^{n}X_{i}\stackrel{\mathcal{D}}{\longrightarrow}\mu  ~~~~~~{\rm{as}}~~ n\rightarrow\infty.
\end{eqnarray}
\end{theorem}

\begin{proof}
Let $Z_{nj}=\frac{X_{j}}{\sqrt{n}}$ and $\mathscr{F}_{nj}=\mathscr{F}_j$, where $j=1,...,n$ and ${n\in\mathbb{N}}$.
Then $\{Z_{nj},{\mathscr{F}_{nj}}\}$ is a martingale difference array of the space $C(\mathcal{S})$.
In the following, the conditions $\boldsymbol{\rm(C1)}$ -$\boldsymbol{\rm(C3)}$ in Theorem $\ref{pro4}$ remain to be verified in turn.\par
\noindent $\boldsymbol{\rm{Step1:}}$
We first show that $E[X^2_j(s)\mid\mathscr{F}_{j-1}]$ is well defined a.s. for any $s\in\mathcal{S}$ and every ${j\in\mathbb{N}}$.
According to $\eqref{eqn21}$, for any given $s_1\in\mathcal{S}$ and $s_1\neq{s_0}$, we have
\begin{eqnarray}
Y_j^2(s_1)\leq\big({M_j}+|Y_j(s_0)|\big)^2
\leq2\bigg[{M^2_j}+Y_j^2(s_0)\bigg]~~~a.s.
\nonumber
\end{eqnarray}
By $\boldsymbol{\rm(A4)}$,
for every ${j\in\mathbb{N}}$, we get
$E[X^2_j(s_0)\mid\mathscr{F}_{j-1}]\leq k^2$.
Thus, together with $\boldsymbol{\rm(A3)}$, 
\begin{eqnarray}
E[X^2_j(s_1)\mid\mathscr{F}_{j-1}]
&\leq&2\bigg({\textstyle\left(\frac{1}{\beta(s_1)}\right)^2}E[M^2_j\mid{\mathscr{F}_{j-1}}]+{\textstyle\left(\frac{\beta(s_0)}{\beta(s_1)}\right)^2}E[X^2_j(s_0)\mid{\mathscr{F}_{j-1}}]\bigg)
\nonumber \\
&\leq& 2\bigg({\textstyle\left(\frac{1}{\beta(s_1)}\right)^2}b+{\textstyle\left(\frac{\beta(s_0)}{\beta(s_1)}\right)^2}k^2\bigg)
\nonumber \\
&<& \infty ~~~~a.s.
\nonumber
\end{eqnarray}
Then, by the arbitrariness of $s_1$, $E[X^2_j(s)\mid\mathscr{F}_{j-1}]$ is well defined for all $s\in\mathcal{S}$ and all ${j\in\mathbb{N}}$.
Therefore, for every $s\in\mathcal{S}$ and ${j\in\mathbb{N}}$, $E[X^2_j(s)]$ is also well defined and $E[X^2_j(s)\mid\mathscr{F}_{j-1}]\leq 2\bigg({\textstyle\left(\frac{1}{\beta(s)}\right)^2}b+{\textstyle\left(\frac{\beta(s_0)}{\beta(s)}\right)^2}k^2\bigg)$ a.s.
By the Cauchy-Schwarz inequality, 
\begin{eqnarray}
E[M_j|X_j(s)|\mid\mathscr{F}_{j-1}]\leq\left(E[M^2_j\mid\mathscr{F}_{j-1}] E[X^2_j(s)\mid\mathscr{F}_{j-1}]\right)^\frac{1}{2}
<\infty~~~~~~~~a.s.
\nonumber
\end{eqnarray}
for any $s\in\mathcal{S}$ and ${j\in\mathbb{N}}$. It follows from $\eqref{eqn21}$ that
\begin{eqnarray}
|Y_j^2(s_1)-Y_j^2(s_0)|\leq{M_j}|Y_j(s_1)+Y_j(s_0)|
\leq{M_j}\left(|Y_j(s_1)|+|Y_j(s_0)|\right)~~~~a.s.
\nonumber
\end{eqnarray}
Therefore, for any $s\in\mathcal{S}$,
\begin{eqnarray}\label{eqn34}
~~&&\bigg|E[Y_j^2(s)\mid\mathscr{F}_{j-1}]-E[Y_j^2(s_0)\mid\mathscr{F}_{j-1}]\bigg|
\nonumber \\
&=&\bigg|E[Y_j^2(s)-Y_j^2(s_0)\mid\mathscr{F}_{j-1}]\bigg|
\nonumber \\
&\leq&E\bigg[\big|Y_j^2(s)-Y_j^2(s_0)\big|\mid\mathscr{F}_{j-1}\bigg]
\nonumber \\
&\leq&E\bigg[{M_j}\left(\big|Y_j(s)\big|+\big|Y_j(s_0)\big|\right)\mid\mathscr{F}_{j-1}\bigg]
\nonumber \\
&=&\bigg(E\big[{M_j}\big|Y_j(s)\big|\mid\mathscr{F}_{j-1}\big]+E\big[{M_j}\big|Y_j(s_0)\big|\mid\mathscr{F}_{j-1}\big]\bigg)
\nonumber \\
&\leq&\big[\left(E[M^2_j\mid\mathscr{F}_{j-1}]E[Y^2_j(s)\mid\mathscr{F}_{j-1}]\right)^\frac{1}{2}+\left(E[M^2_j\mid\mathscr{F}_{j-1}] E[Y^2_j(s_0)\mid\mathscr{F}_{j-1}]\right)^\frac{1}{2}\big]
\nonumber \\
&=&\big(E[M^2_j\mid\mathscr{F}_{j-1}]\big)^\frac{1}{2}\bigg[(E[Y^2_j(s)\mid\mathscr{F}_{j-1}])^\frac{1}{2}+(E[Y^2_j(s_0)\mid\mathscr{F}_{j-1}])^\frac{1}{2}\bigg]
~~~~~~a.s.
\end{eqnarray}

To simplify notation, hereafter we denote
\begin{eqnarray}
&&m_j:=E[M^2_j\mid\mathscr{F}_{j-1}], \quad
w_j(s):=E[X^2_j(s)\mid\mathscr{F}_{j-1}], \quad
W_n(s):=\frac{1}{n}\displaystyle\sum_{j=1}^{n}w_j(s), \quad
\nonumber \\
&&{\rm{and}}~~V_n(s):=E[W_n(s)]=\frac{1}{n}\displaystyle\sum_{j=1}^{n}E[w_j(s)]=\frac{1}{n}\displaystyle\sum_{j=1}^{n}E[X^2_j(s)]
.
\nonumber
\end{eqnarray}
Therefore, $0\leq w_j(s)\leq 2\bigg({\textstyle\left(\frac{1}{\beta(s)}\right)^2}b+{\textstyle\left(\frac{\beta(s_0)}{\beta(s)}\right)^2}k^2\bigg)$ a.s. and $0\leq m_j\leq b$ a.s.
Furthermore, $\eqref{eqn34}$ can be rewritten as
\begin{eqnarray}\label{eqn40}
\left|w_j(s)-{\textstyle\left(\frac{\beta(s_0)}{\beta(s)}\right)^2}w_j(s_0)\right|
&\leq&{\textstyle\frac{1 }{\beta^2(s)}} m_j^\frac{1}{2}\big(w_j(s)^\frac{1}{2}+w_j(s_0)^\frac{1}{2}\big)
\nonumber \\
&\leq&{\textstyle\frac{2 }{\beta^2(s)}}m_j^\frac{1}{2}\left[2\bigg({\textstyle\left(\frac{1}{\beta(s)}\right)^2}b+{\textstyle\left(\frac{\beta(s_0)}{\beta(s)}\right)^2}k^2\bigg)\right]^\frac{1}{2}~~a.s.
\end{eqnarray}
Thus, the condition $\boldsymbol{\rm(C1)}$ is equivalent to the fact that for each $s\in\mathcal{S}$, $W_n(s)$ converges in probability to a positive number as $n$ goes to infinity.
Obviously,
\begin{eqnarray}\label{eqn35}
&&\left|E[w_{j}(s)]-{\textstyle\left(\frac{\beta(s_0)}{\beta(s)}\right)^2}E[w_{j}(s_0)]\right|
\nonumber \\
&\leq& E\left[\big|w_j(s)-{\textstyle\left(\frac{\beta(s_0)}{\beta(s)}\right)^2}w_j(s_0)\big|\right]
\nonumber \\
&\leq&{\textstyle\frac{2 }{\beta^2(s)}}\left[2\bigg({\textstyle\left(\frac{1}{\beta(s)}\right)^2}b+{\textstyle\left(\frac{\beta(s_0)}{\beta(s)}\right)^2}k^2\bigg)\right]^\frac{1}{2}E\big[m_j^\frac{1}{2}\big]
\nonumber \\
&\leq&{\textstyle\frac{2 }{\beta^2(s)}}\left[2\bigg({\textstyle\left(\frac{1}{\beta(s)}\right)^2}b+{\textstyle\left(\frac{\beta(s_0)}{\beta(s)}\right)^2}k^2\bigg)\right]^\frac{1}{2}\big(E[m_j]\big)^\frac{1}{2},~~~~\forall{j\in\mathbb{N}}.
\end{eqnarray}

According to $\boldsymbol{\rm(A5)}$,
we have
\begin{eqnarray}\label{eqn45}
W_n(s_0)\stackrel{P}{\longrightarrow}c.
\end{eqnarray}
On the other hand, by $\boldsymbol{\rm(A4)}$, it is obvious that $W_n(s_0)\leq k^2$ a.s.
Therefore,
\begin{eqnarray}\label{eqn42}
\displaystyle\lim_{n\rightarrow\infty}V_n(s_0)=c>0.
\end{eqnarray}
Similarly, it follows from $\boldsymbol{\rm(A2)}$-$\boldsymbol{\rm(A3)}$ that
$\frac{1}{n}\displaystyle\sum_{j=1}^{n}m_j\stackrel{P}{\longrightarrow}0$
and $\frac{1}{n}\displaystyle\sum_{j=1}^{n}m_j\leq b$ a.s.
Hence
\begin{eqnarray}\label{eqn61}
\displaystyle\lim_{n\rightarrow\infty}\frac{1}{n}\displaystyle\sum_{j=1}^{n}E[m_j]=\displaystyle\lim_{n\rightarrow\infty}\frac{1}{n}\displaystyle\sum_{j=1}^{n}E[M^2_j]=0.
\end{eqnarray}
It is not hard to verify that
\begin{eqnarray}
\frac{1}{n}\displaystyle\sum_{j=1}^{n}m^\frac{1}{2}_j\stackrel{P}{\longrightarrow}0,~~~~~~~~~~~~~~~~~~~~~~~~~~~~~~~~~~\label{eqn44}
\\
{\rm{and}}~~~~~~~~~~~~~~~~~~~~~~~\displaystyle\lim_{n\rightarrow\infty}\frac{1}{n}\displaystyle\sum_{j=1}^{n}(E[m_j])^\frac{1}{2}=0.~~~~~~~~~~~~~~~~~~~~~~~~\label{eqn39}
\end{eqnarray}
Thus, it follows from $\eqref{eqn35}$ and $\eqref{eqn39}$ that for any $s\in\mathcal{S}$,
\begin{eqnarray}\label{eqn43}
\displaystyle\lim_{n\rightarrow\infty}\left|V_n(s)-{\textstyle\left(\frac{\beta(s_0)}{\beta(s)}\right)^2}V_n(s_0)\right|=0.
\end{eqnarray}
Therefore, combining $\eqref{eqn42}$ and $\eqref{eqn43}$, we get
\begin{eqnarray}
\displaystyle\lim_{n\rightarrow\infty}V_n(s)&=&\displaystyle\lim_{n\rightarrow\infty}\left[V_n(s)-{\textstyle\left(\frac{\beta(s_0)}{\beta(s)}\right)^2}V_n(s_0)\right]+\displaystyle\lim_{n\rightarrow\infty}\left[{\textstyle\left(\frac{\beta(s_0)}{\beta(s)}\right)^2}V_n(s_0)\right]
 \nonumber \\
&=&c{\textstyle\left(\frac{\beta(s_0)}{\beta(s)}\right)^2}>0.
 \nonumber
\end{eqnarray}

Let $\varphi(s)=\displaystyle\lim_{n\rightarrow\infty}V_n(s)=c{\textstyle\left(\frac{\beta(s_0)}{\beta(s)}\right)^2}$.
Hence $\varphi: \mathcal{S}\rightarrow\mathbb{R}_+$.
According to $\eqref{eqn40}$ and $\eqref{eqn44}$, for every $s\in\mathcal{S}$,
\begin{eqnarray}
\left|W_n(s)-\varphi(s)\right|&\leq&\left|W_n(s)-{\textstyle\left(\frac{\beta(s_0)}{\beta(s)}\right)^2}W_n(s_0)\right|+ \left|{\textstyle\left(\frac{\beta(s_0)}{\beta(s)}\right)^2}W_n(s_0)-\varphi(s)\right|
 \nonumber \\
&\leq&{\textstyle\frac{2 }{\beta^2(s)}}\left[2\bigg({\textstyle\left(\frac{1}{\beta(s)}\right)^2}b+{\textstyle\left(\frac{\beta(s_0)}{\beta(s)}\right)^2}k^2\bigg)\right]^\frac{1}{2}\left[\frac{1}{n}\displaystyle\sum_{j=1}^{n}m^\frac{1}{2}_j\right]
 \nonumber \\
&&+{\textstyle\left(\frac{\beta(s_0)}{\beta(s)}\right)^2}\left|W_n(s_0)-c\right|.
 \nonumber
\end{eqnarray}
Therefore, it follows from $\eqref{eqn45}$ and $\eqref{eqn44}$ that for a given $s\in\mathcal{S}$,
\begin{eqnarray}
&&~~~{\rm{Prob}}\left(|W_n(s)-\varphi(s)|\geq{\epsilon}\right)
\nonumber \\
&&\leq{\rm{Prob}}\left({\textstyle\frac{2 }{\beta^2(s)}}\left[2\bigg({\textstyle\left(\frac{1}{\beta(s)}\right)^2}b+{\textstyle\left(\frac{\beta(s_0)}{\beta(s)}\right)^2}k^2\bigg)\right]^\frac{1}{2}\left[\frac{1}{n}\displaystyle\sum_{j=1}^{n}m^\frac{1}{2}_j\right]\geq\frac{\epsilon}{2}\right)
\nonumber \\
&&~~+{\rm{Prob}}\left({\textstyle\left(\frac{\beta(s_0)}{\beta(s)}\right)^2}\left|W_n(s_0)-c\right|\geq\frac{\epsilon}{2}\right)
\nonumber \\
&&~~\rightarrow0~~~~~~~~~~~~~~~~~~~~~~~~~~~~~~~~~~~~~~~~~~~~~~~~~~~{\rm{as}}~~n\rightarrow\infty.~
\nonumber
\end{eqnarray}
Thus, the condition $\boldsymbol{\rm(C1)}$ is proved.\par
\noindent $\boldsymbol{\rm{Step2:}}$
Since $Z_{nj}=\frac{X_{j}}{\sqrt{n}}$, we can transform the condition $\boldsymbol{\rm(C2)}$ into the following form:
\begin{eqnarray}\label{eqn62}
\displaystyle\lim_{n\rightarrow\infty}{\rm{Prob}}\left\{\displaystyle\sum_{j=1}^{n}E\left[\frac{1}{n}X^2_{j}(s)I_{\left(|X_{j}(s)|>\epsilon\sqrt{n}\right)}\mid{\mathscr{F}_{j-1}}\right]\geq\epsilon\right\}=0
\end{eqnarray}
for any $\epsilon>0$ and all $s\in\mathcal{S}$.
It follows from $\eqref{eqn21}$ that for any $s\in\mathcal{S}$
\begin{eqnarray}
|X_{j}(s)|\leq \frac{1}{|\beta(s)|}M_j+\left|\frac{\beta(s_0)}{\beta(s)}\right||X_{j}(s_0)|.
\nonumber
\end{eqnarray}
According to $\boldsymbol{\rm(A4)}$, it's easy to verify that for every $s\in\mathcal{S}$ and any $\epsilon>0$,
\begin{eqnarray}
&&~~~~{\rm{Prob}}\left\{\frac{1}{n}\displaystyle\sum_{j=1}^{n}E\left[X^2_{j}(s)I_{\left(|X_{j}(s)|>\epsilon\sqrt{n}\right)}\mid{\mathscr{F}_{j-1}}\right]\geq\epsilon\right\}
\nonumber \\
&&\leq\displaystyle\sum_{j=1}^{n}{\rm{Prob}}\left\{\left(|X_{j}(s)|>\epsilon\sqrt{n}\right)\right\}
\nonumber \\
&&\leq \displaystyle\sum_{j=1}^{n}{\rm{Prob}}\left\{\left(\frac{1}{|\beta(s)|}M_j+\left|\frac{\beta(s_0)}{\beta(s)}\right||X_{j}(s_0)|>\epsilon\sqrt{n}\right)\right\}
\nonumber \\
&&\leq \displaystyle\sum_{j=1}^{n}{\rm{Prob}}\left\{M_j>\epsilon\sqrt{n}|\beta(s)|-\left|{\beta(s_0)}\right|k\right\}.
\nonumber
\end{eqnarray}
Clearly, if $s\in\mathcal{S}$ and $\epsilon>0$ are fixed, then there must exist $N\in\mathbb{N}$ such that for any $n\geq N$,
$\epsilon\sqrt{n}|\beta(s)|>\left|{\beta(s_0)}\right|k$.
Thus, for $n\geq N$, we have
\begin{eqnarray}
\displaystyle\sum_{j=1}^{n}{\rm{Prob}}\left\{M_j>\epsilon\sqrt{n}|\beta(s)|-\left|{\beta(s_0)}\right|k\right\}
&\leq&\displaystyle\sum_{j=1}^{n} \frac{E[M^2_j]}{(\epsilon\sqrt{n}|\beta(s)|-\left|{\beta(s_0)}\right|k)^2}
\nonumber \\
&=&\frac{n}{(\epsilon\sqrt{n}|\beta(s)|-\left|{\beta(s_0)}\right|k)^2}\displaystyle\sum_{j=1}^{n}\frac{E[M^2_j]}{n}.
\nonumber 
\end{eqnarray}
Together with $\eqref{eqn61}$, $\eqref{eqn62}$ is proved.

\noindent $\boldsymbol{\rm{Step3:}}$
It remains to prove the condition $\boldsymbol{\rm(C3)}$, which is equivalent to the following form:
\begin{eqnarray}
\displaystyle\sup_{n\in\mathbb{N}}\frac{1}{n}\displaystyle\sum_{j=1}^{n}E[\|X_{j}\|^2_{C_\rho}]<\infty.
\nonumber
\end{eqnarray}

According to $\eqref{eqn46}$, for any $j\in\mathbb{N}$, we get
\begin{eqnarray}
~~~~~~~~~~~~~~~~~q_\rho(X_j)=\displaystyle\sup_{s_1,s_2\in\mathcal{S}; s_1\neq{s_2}}\frac{|X_j(s_1)-X_j(s_2)|}{\rho(s_1,s_2)}
\leq\varsigma_j~~~~~~~~~~~~~~~~~~~~~{\rm{a.s.}}
\nonumber
\end{eqnarray}
Obviously, for any $j\in\mathbb{N}$,
\begin{eqnarray}
~~~~~~~~~~~~~~~~~~\|X_{j}\|^2_{C_\rho}&=&\left(\max\left\{\displaystyle\sup_{s\in\mathcal{S}}|X_j(s)|, q_\rho(X_{j})\right\}\right)^2
\nonumber \\
&\leq&\left(\displaystyle\sup_{s\in\mathcal{S}}|X_j(s)|+q_\rho(X_{j})\right)^2
\nonumber \\
&\leq&\left(\displaystyle\sup_{s\in\mathcal{S}}|X_j(s)|+\varsigma_j\right)^2
\nonumber \\
&\leq&2\left(\displaystyle\sup_{s\in\mathcal{S}}|X_j(s)|\right)^2+2\varsigma_j^2~~~~~~~~~~~~~~~~~~~~~~~~~{\rm{a.s.}}
\nonumber
\end{eqnarray}
To simplify notation, we denote $\bar{\beta}:=\displaystyle\sup_{s\in\mathcal{S}}|\beta^{-1}(s)|$.
It follows from $\eqref{eqn21}$ that for a given $s\in\mathcal{S}$ and any $j\in\mathbb{N}$,
\begin{eqnarray}
~~~~~~~~~~~~~~~~~~~~~\displaystyle\sup_{s\in\mathcal{S}}|X_j(s)|\leq\bar{\beta}[|\beta(s_0)X_j(s_0)|+M_j]~~~~~~~~~~~~~~~~~~~~~~~~~~~~~{\rm{a.s.}}
\nonumber
\end{eqnarray}
Therefore, for every $j\in\mathbb{N}$,
\begin{eqnarray}
~~~~~~~~~~~~\|X_{j}\|^2_{C_\rho}
&\leq&2\left(\bar{\beta}[|\beta(s_0)X_j(s_0)|+M_j]\right)^2+2\varsigma_j^2
\nonumber \\
&\leq&4\bar{\beta}^2[\beta^2(s_0)X^2_j(s_0)+M_j^2]+2\varsigma_j^2~~~~~~~~~~~~~~~~~~~~~{\rm{a.s.}}
\nonumber
\end{eqnarray}
According to $\boldsymbol{\rm(A3)}$-$\boldsymbol{\rm(A4)}$, we have
\begin{eqnarray}
~~~~~~~~E[M^2_j]=E\left[E[M^2_j\mid{\mathscr{F}_{j-1}}]\right]\leq{b}
~~~~{\rm{a.s.}}~~~~{\rm{and}}~~~~E[X^2_j(s_0)]=\leq k^2 ~~~~{\rm{a.s.}}
\nonumber
\end{eqnarray}
for any $j\in\mathbb{N}$.
Furthermore, we get
\begin{eqnarray}
E[\|X_{j}\|^2_{C_\rho}]&\leq&E\left[4\bar{\beta}^2[\beta^2(s_0)X^2_j(s_0)+M_j^2]+2\varsigma_j^2\right]
\nonumber \\
&\leq&4\bar{\beta}^2(k^2\beta^2(s_0)+b)+2E\left[\varsigma_j^2\right].
\nonumber
\end{eqnarray}
Then, it follows from $\boldsymbol{\rm(A6)}$ that
\begin{eqnarray}
\displaystyle\sup_{n\in\mathbb{N}}\frac{1}{n}\displaystyle\sum_{j=1}^{n}E[\|X_{j}\|^2_{C_\rho}]
\leq4\bar{\beta}^2(k^2\beta^2(s_0)+b)+2\displaystyle\sup_{n\in\mathbb{N}}\frac{1}{n}\displaystyle\sum_{j=1}^{n}E\left[\varsigma_j^2\right]<\infty.
\nonumber
\end{eqnarray}
Thus, applying Theorem $\ref{pro4}$, the proof is complete.
\end{proof}

In the following, we show two results that are extensions of Theorem $\ref{cor1}$.
\begin{corollary}\label{cor4}
Let $\{X_i, \mathscr{F}_i\}_{i=1}^{\infty}$ be a martingale difference sequence on $C(\mathcal{S})$. Suppose that $\boldsymbol{\rm(A1)}$ and $\boldsymbol{\rm(A3)}$-$\boldsymbol{\rm(A6)}$ in Theorem $\ref{cor1}$ are satisfied. Assume that

\noindent $\boldsymbol{\rm(A2')}$ $\displaystyle\lim_{n\rightarrow\infty}\frac{1}{n}\displaystyle\sum_{i=1}^{n}E[M^2_i]=0$.

\noindent Then, $\eqref{eqn47}$ holds.
\end{corollary}
\begin{proof}
Since $E[M^2_i\mid{\mathscr{F}_{i-1}}]\geq0$, we have
\begin{eqnarray}
\frac{1}{n}\displaystyle\sum_{i=1}^{n}E[M^2_i]=E\left[\frac{1}{n}\displaystyle\sum_{i=1}^{n}E[M^2_i\mid{\mathscr{F}_{i-1}}]\right]=E\left[\left|\frac{1}{n}\displaystyle\sum_{i=1}^{n}E[M^2_i\mid{\mathscr{F}_{i-1}}]\right|\right].
\nonumber
\end{eqnarray}
This implies that $\frac{1}{n}\displaystyle\sum_{i=1}^{n}E[M^2_i\mid{\mathscr{F}_{i-1}}]\stackrel{\mathcal{L}_1}{\longrightarrow}0$.
Thus,
$\frac{1}{n}\displaystyle\sum_{i=1}^{n}E[M^2_i\mid{\mathscr{F}_{i-1}}]\stackrel{P}{\longrightarrow}0$.
This shows that $\boldsymbol{\rm(A2')}$ and $\boldsymbol{\rm(A2)}$ are equivalent.
\end{proof}

\begin{corollary}\label{cor5}
Let $\{X_i, \mathscr{F}_i\}_{i=1}^{\infty}$ be a martingale difference sequence on $C(\mathcal{S})$. Suppose that $\boldsymbol{\rm(A4)}$ and $\boldsymbol{\rm(A5)}$ in Theorem $\ref{cor1}$ are satisfied. Assume that

\noindent $\boldsymbol{\rm(A6')}$
There exists a real nonnegative random sequence $\{\varsigma_i\}$ on $(\Omega,\mathscr{F},P)$ such that
$\displaystyle\lim_{n\rightarrow\infty}\frac{1}{n}\displaystyle\sum_{i=1}^{n}E[\varsigma^2_i]=0$
and $\eqref{eqn46}$ holds for any $i\in\mathbb{N}$.

\noindent $\boldsymbol{\rm(A7)}$ For any $i\in\mathbb{N}$, there exists a constant $\bar{b}>0$ such that $E[\varsigma^2_i\mid{\mathscr{F}_{i-1}}]\leq\bar{b}$ a.s.

\noindent Then, $\eqref{eqn47}$ holds.
\end{corollary}
\begin{proof}
Let $\beta(s)=\hat{\beta}$, where $\hat{\beta}$ is an arbitrary nonzero constant.
Since $\mathcal{S}$ is compact, $\displaystyle\sup_{s_1,s_2\in\mathcal{S}}\rho(s_1,s_2)<\infty$.
To simplify notation, we denote $\zeta=\displaystyle\sup_{s_1,s_2\in\mathcal{S}}\rho(s_1,s_2)$.
Moreover, let $M_i={\varsigma_i}\zeta$ and  $Y_i(s)=|\hat{\beta}X_i(s)|$.
Hence $\boldsymbol{\rm(A1)}$ is obvious.
By $\boldsymbol{\rm(A6')}$,
\begin{eqnarray}
\displaystyle\lim_{n\rightarrow\infty}\frac{1}{n}\displaystyle\sum_{i=1}^{n}E[M^2_i]=\zeta^2 \displaystyle\lim_{n\rightarrow\infty}\frac{1}{n}\displaystyle\sum_{i=1}^{n}E[\varsigma^2_i]=0.
\nonumber
\end{eqnarray}
It follows from $\boldsymbol{\rm(A7)}$ that
\begin{eqnarray}
E[M^2_i\mid{\mathscr{F}_{i-1}}]=\zeta^2 E[\varsigma^2_i\mid{\mathscr{F}_{i-1}}]\leq\zeta^2\bar{b}.
\nonumber
\end{eqnarray}
Thus, $\boldsymbol{\rm(A2')}$ and $\boldsymbol{\rm(A3)}$ are verified.
$\boldsymbol{\rm(A6)}$ can be directly derived from $\boldsymbol{\rm(A6')}$.
Applying Theorem $\ref{cor1}$ and Corollary $\ref{cor4}$, the proof is complete.
\end{proof}

\subsection{Main Results}
\label{sec4.2}
We discuss the main theorems in this subsection,
that is, the asymptotics of the optimal value for SAA with AMIS.
Here, for consistency of notation, we use $C(\mathcal{X})$ to stand for the space of continuous functions on $\mathcal{X}$, equipped with the sup-norm.
Just like before, $\mathcal{T}=\displaystyle\arg\min_{x\in\mathcal{X}}f(x)$.
\begin{theorem}\label{the7}
Let $\mathcal{X}$ be a compact set on ${\mathbb{R}^n}$ and Assumptions $\ref{ass1}$-$\ref{ass2}$ hold.
Assume that the following statements hold.\par
\noindent $\boldsymbol{\rm(B1)}$
For some $x_0\in\mathcal{X}$, Assumptions $\ref{ass3}$-$\ref{ass4}$ are satisfied and there exists a constant $\tilde{c}>f^2(x_0)$ such that
\begin{eqnarray}\label{eqn31}
\frac{1}{n}\displaystyle\sum_{i=1}^{n}E\left[\bigg(\frac{F(x_0,\theta_i)}{\psi_i(\theta_i)}\bigg)^2\middle\vert\,{\mathscr{G}_{i-1}}\right]\stackrel{P}{\longrightarrow}\tilde{c}.
\nonumber
\end{eqnarray}

\noindent $\boldsymbol{\rm(B2)}$
There exists an integrable function ${\alpha}:\Theta\rightarrow\mathbb{R}_+$ such that for any $x_1,x_2\in\mathcal{X}$
\begin{eqnarray}\label{eqn26}
|F(x_1,\theta)-F(x_2,\theta)|\leq{\alpha(\theta)}\|x_1-x_2\| ~~{\rm{a.s.}}
\end{eqnarray}
with $\displaystyle\sup_{n\in\mathbb{N}}\frac{1}{n}\displaystyle\sum_{i=1}^{n}E\left[\left(\frac{\alpha(\theta)}{\psi_i(\theta)}\right)^2\right]<\infty$.

\noindent $\boldsymbol{\rm(B3)}$
There exists a real nonnegative random sequence $\{\mathcal{A}_i\}$ on $(\Omega,\mathscr{F},P)$ and a function $\mathcal{V}:\mathcal{X}\rightarrow\mathbb{R}_+$  such that for any $i\in\mathbb{N}$ and any $x\in\mathcal{X}$,
\begin{eqnarray}
|F(x,\theta)-f(x)\psi_i(\theta)|\leq\mathcal{V}(x){\mathcal{A}_i(\theta)}~~{\rm{a.s.}},&&\label{eqn60}
\\
\frac{1}{n}\displaystyle\sum_{i=1}^{n}E\left[\bigg(\frac{\mathcal{A}_i(\theta_i)}{\psi_i(\theta_i)}\bigg)^2\middle\vert\,{\mathscr{G}_{i-1}}\right]\stackrel{P}{\longrightarrow}0,&&
\nonumber
\\
{\rm{and}}~~~~~~~~~~~~~~~~~~~~~~~~~~~~~~~
E\left[\bigg(\frac{\mathcal{A}_i(\theta_i)}{\psi_i(\theta_i)}\bigg)^2\middle\vert\,{\mathscr{G}_{i-1}}\right]\leq\tilde{b}~~{\rm{a.s.}},&&~~~~~~~~~~~~~~~~~~~~~~~~~~
\nonumber
\end{eqnarray}
where $\tilde{b}$ is a positive constant and $\displaystyle\sup_{x\in\mathcal{X}}\mathcal{V}(x)<\infty$..

\noindent Then
\begin{eqnarray}
&&\vartheta_n=\inf_{x\in\mathcal{T}}f_n(x)+o_p(\textstyle{\frac{1}{\sqrt{n}}}),\label{eqn27}~~~~~~~~~~~~~~~~~~~
\\
&&{\sqrt{n}}(\vartheta_n-\vartheta)\stackrel{\mathcal{D}}{\longrightarrow}\displaystyle\inf_{x\in\mathcal{T}}Y(x),\label{eqn48}  ~~~~~~{\rm{as}}~~ n\rightarrow\infty,
\end{eqnarray}
where $Y(x)\sim\mathcal{N}(0,\sigma^2(x))$ , i.e., $Y(x)$ is a normal distribution with mean $0$ and variance $\sigma^2(x)=[\tilde{c}-f^2(x_0)]\frac{\mathcal{V}^2(x)}{\mathcal{V}^2(x_0)}$.
In particular, if $\mathcal{T}=\{\hat{x}\}$ is a singleton, then
\begin{eqnarray}\label{eqn28}
~~~~~~{\sqrt{n}}(\vartheta_n-\vartheta)\stackrel{\mathcal{D}}{\longrightarrow}\mathcal{N}(0,\sigma^2(\hat{x})),  ~~~~~~{\rm{as}}~~ n\rightarrow\infty.
\end{eqnarray}
\end{theorem}

\begin{proof}
The idea of the proof is similar to that of Theorem 5.7 in $\cite{7}$. The associated functional CLT and Delta Theorem need to be used.\par
Obviously, for any $x\in\mathcal{X}$, the expectation function $f(x)$ is finite and $\{\Upsilon_i(x), \mathscr{G}_i\}$ is a martingale difference sequence by Lemma $\ref{lem2}$.
In the following, we show that $\Upsilon_i$ satisfies the conditions in Theorem $\ref{cor1}$.

First, integrating over $\eqref{eqn26}$, we get
\begin{eqnarray}\label{eqn30}
~~~~~~~~|f(x_1)-f(x_2)|\leq\tilde{\alpha}\|x_1-x_2\|,~~~~\forall x_1,x_2\in\mathcal{X},
\end{eqnarray}
where
$\tilde{\alpha}:=\displaystyle\int_{\Theta}{\alpha(\theta)}{\mathrm{d}{\theta}}$. Obviously, $\tilde{\alpha}>0$.
Therefore, $f(x)$ is Lipschitz continuous on the compact set $\mathcal{X}$.
This implies that $\vartheta$ and $\mathcal{T}$ are nonempty.
Obviously, we have $f\in C(\mathcal{X})$ and $F(\cdot,\theta)\in C(\mathcal{X})$ for any given $\theta\in\Theta$.
In fact, it is easy to see that
\begin{eqnarray}
\Upsilon_i(x)=\begin{cases}\frac{F(x,\theta_i)}{\psi_i(\theta_i)}-f(x),~~~~~~\theta_i\in{\Theta},
\\~~~-f(x),~~~~~~~~~{\theta_i\in{\Theta_i}\backslash{\Theta}}.
\end{cases}
\nonumber
\end{eqnarray}
Therefore, $\Upsilon_i\in C(\mathcal{X})$.
Then, $f_n\in C(\mathcal{X})$ and $H_n=f_n-f\in C(\mathcal{X})$.\par
It follows from $\eqref{eqn26}$ and $\eqref{eqn30}$ that
\begin{eqnarray}
|\Upsilon_i(x_1)-\Upsilon_i(x_2)|&=&\left|\frac{F(x_1,\theta_i)}{\psi_i(\theta_i)}-f(x_1)-\frac{F(x_2,\theta_i)}{\psi_i(\theta_i)}+f(x_2)\right|
\nonumber \\
&\leq&\left|\frac{F(x_1,\theta_i)}{\psi_i(\theta_i)}-\frac{F(x_2,\theta_i)}{\psi_i(\theta_i)}\right|+|f(x_1)-f(x_2)|
\nonumber \\
&\leq&\left(\frac{\alpha(\theta_i)}{\psi_i(\theta_i)}+\tilde{\alpha}\right)\|x_1-x_2\|~~~~~~{\rm{a.s.}}
\nonumber
\end{eqnarray}
Denote $\varsigma_i=\frac{\alpha(\theta_i)}{\psi_i(\theta_i)}+\tilde{\alpha}$.
Obviously, $\varsigma_i>0$ a.s., so $\eqref{eqn46}$ in Theorem $\ref{cor1}$ is satisfied.
Further, we have
\begin{eqnarray}
\displaystyle\sup_{n\in\mathbb{N}}\frac{1}{n}\displaystyle\sum_{i=1}^{n}E\left[\varsigma_i^2\right]\leq
\displaystyle\sup_{n\in\mathbb{N}}\frac{2}{n}\displaystyle\sum_{i=1}^{n}E\left[\left(\frac{\alpha(\theta)}{\psi_i(\theta)}\right)^2\right]+2\tilde{\alpha}^2<\infty.
\nonumber
\end{eqnarray}
Therefore, $\boldsymbol{\rm(A6)}$ is verified.

Since for $x_0\in\mathcal{X}$ and all $i\in\mathbb{N}$,
\begin{eqnarray}
E[\Upsilon^2_i(x_0)\mid{\mathscr{G}_{i-1}}]&=&\displaystyle\int_{\Theta_i}{\Upsilon^2_i(x_0)}{\psi_{i}(\theta_{i})}{\mathrm{d}{\theta_{i}}}
\nonumber \\
&=&\displaystyle\int_{\Theta_i}{\left[{\frac{F(x_0,\theta_{i})}{\psi_{i}(\theta_{i})}-f(x_0)}\right]^2}{\psi_{i}(\theta_{i})}{\mathrm{d}{\theta_{i}}}
\nonumber \\
&=&\displaystyle\int_{\Theta_i}\left[\frac{F^2(x_0,\theta_{i})}{\psi_{i}(\theta_{i})}-2f(x_0){F(x_0,\theta_{i})}+f^2(x_0){\psi_{i}(\theta_{i})}\right]{\mathrm{d}{\theta_{i}}}
\nonumber \\
&=&E\left[\bigg(\frac{F(x_0,\theta_i)}{\psi_i(\theta_i)}\bigg)^2\middle\vert\,{\mathscr{G}_{i-1}}\right]-f^2(x_0)~~~~~~{\rm{a.s.}},
\nonumber
\end{eqnarray}
combined with $\boldsymbol{\rm(B1)}$, it is not difficult to verify that $E[\Upsilon^2_i(x_0)\mid{\mathscr{F}_{i-1}}]$  satisfies $\boldsymbol{\rm(A5)}$. 
Moreover, $\Upsilon_i(x_0)$ satisfies $\boldsymbol{\rm(A4)}$ by Lemma $\ref{lem3}$.

Let $M_i=\frac{\mathcal{A}_i(\theta_i)}{\psi_i(\theta_i)}$.
According to $\eqref{eqn60}$, we have
\begin{eqnarray}
\left|\frac{|\Upsilon_i(x_1)|}{\mathcal{V}(x_1)}-\frac{|\Upsilon_i(x_2)|}{\mathcal{V}(x_2)}\right|\leq M_i~~{\rm{a.s.}}
\nonumber
\end{eqnarray}
Thus, $\boldsymbol{\rm(A1)}$-$\boldsymbol{\rm(A3)}$ are verified from $\boldsymbol{\rm(B3)}$.
From the above discussion, $\Upsilon_i$ satisfies the conditions in Theorem $\ref{cor1}$.
Then, applying Theorem $\ref{cor1}$, there exists a Gaussian measure $Y$ on $C(\mathcal{X})$ such that for any $x\in\mathcal{X}$,
\begin{eqnarray}\label{eqn50}
~~~~~~~~~~\frac{1}{\sqrt{n}}\displaystyle\sum_{i=1}^{n}\Upsilon_{i}(x)={\sqrt{n}}H_n(x)\stackrel{\mathcal{D}}{\longrightarrow}Y(x)  ~~~~~~{\rm{as}}~~ n\rightarrow\infty.
\end{eqnarray}
According to the properties of martingale difference sequences, it follows that $Y(x)\sim\mathcal{N}(0,\sigma^2(x))$ (see Remark $\ref{rem4}$).

Let $R(V):=\displaystyle\inf_{x\in\mathcal{X}}V(x)$, where $V\in C(\mathcal{X})$.
Since $\mathcal{X}$ is compact, $R(\cdot)$ is a real value and measurable (with respect to the Borel $\sigma$-algebras) on $C(\mathcal{X})$.
It can be verified that $R(\cdot)$ is Lipschitz continuous.
Specifically, since for any $V_1,V_2\in C(\mathcal{X})$, there must exist $x_1,x_2\in\mathcal{X}$ such that $V_1(x_1)=\displaystyle\inf_{x\in\mathcal{X}}V_1(x)$
and $V_2(x_2)=\displaystyle\inf_{x\in\mathcal{X}}V_2(x)$.
Thus,
\begin{eqnarray}
|R(V_1)-R(V_2)|&=&\left|\displaystyle\inf_{x\in\mathcal{X}}V_1(x)-\displaystyle\inf_{x\in\mathcal{X}}V_2(x)\right|
\nonumber \\
&\leq&\max\{\left|V_1(x_1)-V_2(x_1)\right|,\left|V_1(x_2)-V_2(x_2)\right|\}
\nonumber \\
&\leq&\displaystyle\sup_{x\in\mathcal{X}}\left|V_1(x)-V_2(x)\right|
\nonumber \\
&=&\|V_1-V_2\|.
\nonumber
\end{eqnarray}
On the other hand, $V(x)$ is differentiable with respect to itself for every $x\in\mathcal{X}$.
Then, according to Corollary $\ref{cor2}$, $R(\cdot)$ is directionally differentiable.
Furthermore, for any given $\Gamma\in{C(\mathcal{X})}$, the directional derivative of $R(V)$ at $\Gamma$ is
\begin{eqnarray}
R'(V,\Gamma)=\displaystyle\inf_{x\in\tilde{\Lambda}(\Gamma)}V(x),
\nonumber
\end{eqnarray}
where $\tilde{\Lambda}(\Gamma):=\displaystyle\arg\min_{x\in\mathcal{X}}\Gamma(x)$ and $V\in C(\mathcal{X})$ is arbitrary.
Thus, by Proposition $\ref{pro3}$, $R(\cdot)$ is Hadamard directionally differentiable, and $R'_{\mathscr{H}}(V,\Gamma)=R'(V,\Gamma)$.\par
\par
As discussed above,
the mapping $R: C(\mathcal{X})\rightarrow{\mathbb{R}}$ is Hadamard directionally differentiable at $Y\in{C(\mathcal{X})}$, and $R'_{\mathscr{H}}(f,Y)=\displaystyle\inf_{x\in\tilde{\Lambda}(f)}Y(x)=\displaystyle\inf_{x\in\mathcal{T}}Y(x)$.
Further, $\sqrt{n}$ and $f_n-f\in{C(\mathcal{X})}$ satisfy $\sqrt{n}\rightarrow\infty$ and $\sqrt{n}(f_n-f)\stackrel{\mathcal{D}}{\longrightarrow}Y$ as $n\rightarrow\infty$, respectively.
It follows from the definition of $R(\cdot)$ and $\tilde{\Lambda}(\cdot)$ that
$\vartheta=R(f)$, $\vartheta_n=R(f_n)$ and $\mathcal{T}=\tilde{\Lambda}(f)$.
Thus, applying Theorem $\ref{the3}$, we have
\begin{eqnarray}
&&{\sqrt{n}}(\vartheta_n-\vartheta)\stackrel{\mathcal{D}}{\longrightarrow}\displaystyle\inf_{x\in\mathcal{T}}Y(x),  ~~~~~~{\rm{as}}~~ n\rightarrow\infty,
\nonumber \\
\rm{and}~~~~~~~~~~~~~~~~~~~&&{\sqrt{n}}(\vartheta_n-\vartheta)=\inf_{x\in\mathcal{T}}{\sqrt{n}}\big(f_n(x)-f(x)\big)+o_p(\textstyle{1}).~~~~~~~~~~~~~~~~~~~~~~~~
\nonumber
\end{eqnarray}
In addition, since $\vartheta=R(f)=\displaystyle\inf_{x\in\mathcal{X}}f(x)=\displaystyle\inf_{x\in\mathcal{T}}f(x)$, we get
\begin{eqnarray}
{\sqrt{n}}\vartheta_n=\inf_{x\in\mathcal{T}}{\sqrt{n}}f_n(x)+o_p(\textstyle{1}).
\nonumber
\end{eqnarray}
Therefore, $\eqref{eqn27}$ is proved.\par
Finally, if $\mathcal{T}=\{\hat{x}\}$ is a singleton, then $\displaystyle\inf_{x\in\mathcal{T}}Y(\hat{x})=\mathcal{N}(0,\sigma^2(\hat{x}))$.
Thus, $\eqref{eqn28}$ follows from $\eqref{eqn48}$.
The proof is complete.
\end{proof}

A brief remark about $Y(x)\sim\mathcal{N}(0,\sigma^2(x))$ in the above theorem is given below.
Those who are familiar with martingale difference sequences can skip it.
\begin{remark}\label{rem4}
Since $\{\Upsilon_i(x), \mathscr{G}_i\}$ is a martingale difference sequence for a given $x\in\mathcal{X}$, by definition,
$E[\Upsilon_i(x)\mid{\mathscr{G}_{i-1}}]=0$.
By the law of total expectation, we obtain
\begin{eqnarray}
E\left[\frac{1}{\sqrt{n}}\displaystyle\sum_{i=1}^{n}\Upsilon_i(x)\right]=\frac{1}{\sqrt{n}}\displaystyle\sum_{i=1}^{n}E[\Upsilon_i(x)]=\frac{1}{\sqrt{n}}\displaystyle\sum_{i=1}^{n}E\left[E[\Upsilon_i(x)\mid{\mathscr{G}_{i-1}}]\right]=0,
\nonumber
\end{eqnarray}
and for any $i,j\in\mathbb{N}$ with $i>j$,
\begin{eqnarray}
E[\Upsilon_i(x)\Upsilon_j(x)]=E\left[E[\Upsilon_i(x)\Upsilon_j(x)\mid{\mathscr{G}_{j}}]\right]=E\left[\Upsilon_j(x)E[\Upsilon_i(x)\mid{\mathscr{G}_{j}}]\right]=0.
\nonumber
\end{eqnarray}
Furthermore,
\begin{eqnarray}
{\rm{Var}}[\Upsilon_j(x)]=E[\Upsilon^2_j(x)],&&~~~~~\forall{j\in\mathbb{N}},
\nonumber
\\
{\rm{Var}}[\Upsilon_i(x)+\Upsilon_j(x)]=E[\Upsilon^2_i(x)]+E[\Upsilon^2_j(x)],&&~~~~~\forall{i,j\in\mathbb{N}}~{\rm{and}}~i\neq{j}.
\nonumber
\end{eqnarray}
Hence
\begin{eqnarray}
{\rm{Var}}\left[\frac{1}{\sqrt{n}}\displaystyle\sum_{i=1}^{n}\Upsilon_i(x)\right]=\frac{1}{n}\displaystyle\sum_{i=1}^{n}{\rm{Var}}[\Upsilon_i(x)]=\frac{1}{n}\displaystyle\sum_{i=1}^{n}E[\Upsilon^2_i(x)].
\nonumber
\end{eqnarray}
We claim that for any $x\in\mathcal{X}$, $\sigma^2(x)>0$.
In fact, $\sigma^2$ is equivalent to $\varphi$ in Theorem $\ref{cor1}$.
By the proof of the convergence of $\{V_n(s)\}_{n=1}^\infty$ in $\boldsymbol{\rm{Step1}}$ of Theorem $\ref{cor1}$, we have $\sigma^2(x)=\displaystyle\lim_{n\rightarrow\infty}\frac{1}{n}\displaystyle\sum_{i=1}^{n}E[\Upsilon^2_i(x)]=[\tilde{c}-f^2(x_0)]\frac{\mathcal{V}^2(x)}{\mathcal{V}^2(x_0)}>0$.
Therefore, $E[Y(x)]=0$ and ${\rm{Var}}[Y(x)]=\sigma^2(x)$.
Notice that $Y$ is a Gaussian measure on $C(\mathcal{X})$.
Then, $Y(x)\sim\mathcal{N}(0,\sigma^2(x))$.
\end{remark}

We obtain the following result by slightly modifying $\boldsymbol{\rm(B3)}$.
\begin{corollary}\label{cor6}
Suppose that all assumptions are as in Theorem $\ref{the7}$. Only replace $\boldsymbol{\rm(B3)}$ by the following condition.\par
\noindent $\boldsymbol{\rm(B3')}$
$f(x)\neq0$ for all $x\in\mathcal{X}$ and
there exists a real nonnegative random sequence $\{\mathcal{M}_i\}$ on $(\Omega,\mathscr{F},P)$ such that for any $i\in\mathbb{N}$ and $x_1,x_2\in\mathcal{X}$,
\begin{eqnarray}
\left|\frac{F(x_1,\theta_i)}{f(x_1)\psi_i(\theta_i)}-\frac{F(x_2,\theta_i)}{f(x_2)\psi_i(\theta_i)}\right|\leq\mathcal{M}_i~~{\rm{a.s.}}
\nonumber
\end{eqnarray}
with $\frac{1}{n}\displaystyle\sum_{i=1}^{n}E[\mathcal{M}^2_i\mid{\mathscr{G}_{i-1}}]\stackrel{P}{\longrightarrow}0$ and
$E[\mathcal{M}^2_i\mid{\mathscr{G}_{i-1}}]\leq b$ a.s., where $b$ is a positive constant.

\noindent
Then, $\eqref{eqn27}$ and $\eqref{eqn48}$ hold.
\end{corollary}

\begin{proof}
Since $\mathcal{X}$ is a compact set, combined with $\eqref{eqn30}$, we have
$\displaystyle\sup_{x\in\mathcal{X}}|f(x)|<\infty$.
Let $\beta(x)=\frac{1}{f(x)}$.
Then for any $i\in\mathbb{N}$ and $x_1,x_2\in\mathcal{X}$,
\begin{eqnarray}
\left||\beta(x_1)\Upsilon_i(x_1)|-|\beta(x_2)\Upsilon_i(x_2)|\right|&\leq&\left|\beta(x_1)\Upsilon_i(x_1)-\beta(x_2)\Upsilon_i(x_2)\right|
\nonumber \\
&\leq&\left|\frac{F(x_1,\theta_i)}{f(x_1)\psi_i(\theta_i)}-\frac{F(x_2,\theta_i)}{f(x_2)\psi_i(\theta_i)}\right|
\nonumber \\
&\leq&\mathcal{M}_i ~~{\rm{a.s.}}
\nonumber
\end{eqnarray}
According to $\boldsymbol{\rm(B3')}$, $\boldsymbol{\rm(A1)}$-$\boldsymbol{\rm(A3)}$ are verified.
Therefore, Theorem $\ref{cor1}$ can still be used.
An additional note is that $\sigma^2(x)=[\tilde{c}-f^2(x_0)]{\textstyle\left(\frac{f(x)}{f(x_0)}\right)^2}>0$.
The rest of the proof is the same as in Theorem $\ref{the7}$.
\end{proof}

The following theorem also yields a similar result of Theorem $\ref{the7}$.
The difference is that it uses another result of the functional CLT for martingale difference sequences.
\begin{theorem}\label{the8}
Let $\mathcal{X}$ be a compact set on ${\mathbb{R}^n}$ and Assumptions $\ref{ass1}$-$\ref{ass2}$ hold.
Assume that $\boldsymbol{\rm(B1)}$ is satisfied and the following statements hold.\par
\noindent $\boldsymbol{\rm(B2')}$
There exists an integrable function ${\alpha}:\Theta\rightarrow\mathbb{R}_+$ such that for any $x_1,x_2\in\mathcal{X}$, $\eqref{eqn26}$ holds.

\noindent $\boldsymbol{\rm(B4)}$
There exists a real nonnegative random sequence $\{\mathcal{\bar{A}}_i\}$ on $(\Omega,\mathscr{F},P)$ such that for any $i\in\mathbb{N}$ and $x_1,x_2\in\mathcal{X}$,
\begin{eqnarray}\label{eqn51}
~~~~~~|F(x_1,\theta)-f(x_1)\psi_i(\theta)-F(x_2,\theta)+f(x_2)\psi_i(\theta)|\leq{\mathcal{\bar{A}}_i(\theta)}\|x_1-x_2\|~~{\rm{a.s.}},
\end{eqnarray}
with $\displaystyle\lim_{n\rightarrow\infty}\frac{1}{n}\displaystyle\sum_{i=1}^{n}E\left[\left(\frac{{\mathcal{\bar{A}}_i(\theta)}}{\psi_i(\theta)}\right)^2\right]=0.
$
Moreover, there exists a constant $\hat{b}>0$ such that
\begin{eqnarray}
E\left[\bigg(\frac{{\mathcal{\bar{A}}_i(\theta)}}{\psi_i(\theta)}\bigg)^2\middle\vert\,{\mathscr{G}_{i-1}}\right]\leq\hat{b}~~{\rm{a.s.}}
\nonumber
\end{eqnarray}

\noindent
Then, $\eqref{eqn27}$ and $\eqref{eqn48}$ hold.
\end{theorem}
\begin{proof}
The proof of this theorem is similar to the proof of Theorem $\ref{the7}$. The difference lies in the application of Corollary $\ref{cor5}$ to prove the functional CLT.
To be specific, $\eqref{eqn26}$ ensures that $\Upsilon_i\in C(\mathcal{X})$ and $f(x)$ is Lipschitz continuous.
According to $\eqref{eqn51}$,
we get
\begin{eqnarray}
|\Upsilon_i(x_1)-\Upsilon_i(x_2)|
\leq\frac{\mathcal{\bar{A}}_i(\theta)}{\psi_i(\theta)}\|x_1-x_2\|~~{\rm{a.s.}}
\nonumber
\end{eqnarray}
Let $\varsigma_i=\frac{\mathcal{\bar{A}}_i(\theta)}{\psi_i(\theta_i)}$.
Clearly, $\boldsymbol{\rm(A6')}$ and $\boldsymbol{\rm(A7)}$ are verified from $\boldsymbol{\rm(B2')}$ and $\boldsymbol{\rm(B4)}$, respectively.
$\boldsymbol{\rm(B1)}$ can derive $\boldsymbol{\rm(A4)}$ and $\boldsymbol{\rm(A5)}$, which is proved in Theorem $\ref{the7}$.
Thus, using Corollary 4.3, we have $\eqref{eqn50}$.

The rest of the proof is treated in the same way as in Theorem $\ref{the7}$.
The proof is complete.
\end{proof}

\section{Conclusion}
\label{sec5}
In this paper, we mainly focus on the stochastic optimization problem $\eqref{eqn24}$ and its counterpart problem for SAA with AMIS $\eqref{eqn25}$.
Over the past decade, the method of SAA with AMIS has been more favored.
Here, we investigate uniform exponential convergence of SAA with AMIS and asymptotics of its optimal value.
The former study draws on previous ideas from econometricians.
We establish AMMC that ensures the SE condition and obtain pointwise convergence using a concentration inequality for bounded martingale differences.
In this way, we obtain a new exponential convergence rate in Theorem $\ref{the1}$.
In order to study the asymptotics, we first derive a suitable functional CLT for martingale difference sequences (Theorem $\ref{cor1}$), by applying a CLT for martingale difference arrays.
And then Corollaries $\ref{cor4}$-$\ref{cor5}$ follow.
Finally, combined with the Delta theorem, the asymptotics of the optimal values for SAA with AMIS are proved, see Theorems $\ref{the7}$-$\ref{the8}$ and Corollary $\ref{cor6}$.

The use of the CLT for martingale difference arrays may not be thorough enough.
Extending the conditions of Theorem $\ref{cor1}$ will be our future work.
In addition, because of the good performance of SAA with AMIS, we expect it to be used in more stochastic programs.

\bibliographystyle{siamplain}
\bibliography{references}
\end{document}